\documentclass[a4paper,11pt]{amsart}
\usepackage{amsmath,amssymb,amsthm,graphicx}

\newcommand{\IM}{\operatorname{\mathcal{M}}}
\newcommand{\BB}{\operatorname{\mathcal{B}}}
\newcommand{\EE}{\operatorname{\mathcal{E}}}
\newcommand{\LL}{\operatorname{\mathcal{L}}}

\newcommand{\IA}{\operatorname{\mathcal{A}}}
\newcommand{\Si}{\dot{S}}

\renewcommand{\H}{\mathcal{H}}

\newcommand{\cyl}{\operatorname{cyl}}

\newcommand{\eps}{\epsilon}

\newcommand{\CR}{\bar{\partial}}

\newcommand{\ev}{\operatorname{ev}}

\newcommand{\cst}{\operatorname{const}}

\newcommand{\CZ}{\operatorname{CZ}}
\newcommand{\QH}{\operatorname{QH}}
\newcommand{\HF}{\operatorname{HF}}
\newcommand{\HC}{\operatorname{HC}}
\newcommand{\CF}{\operatorname{CF}}
\newcommand{\SH}{\operatorname{SH}}
\newcommand{\IBL}{\operatorname{IBL}}
\renewcommand{\H}{\operatorname{H}}

\newcommand{\ft}{\operatorname{ft}}

\newcommand{\ind}{\operatorname{ind}}

\newcommand{\del}{\partial}
\newcommand{\RS}{\IR \times S^1}

\newcommand{\IC}{\operatorname{\mathbb{C}}}
\newcommand{\IZ}{\operatorname{\mathbb{Z}}}
\newcommand{\IQ}{\operatorname{\mathbb{Q}}}
\newcommand{\IR}{\operatorname{\mathbb{R}}}
\newcommand{\IN}{\operatorname{\mathbb{N}}}

\newcommand{\Coker}{\operatorname{Coker}}
\newcommand{\Ker}{\operatorname{Ker}}
\newcommand{\Q}{\operatorname{Q}}

\newcommand{\T}{\operatorname{\mathcal{T}}}
\renewcommand{\LL}{\operatorname{\mathcal{L}}}
\renewcommand{\P}{\operatorname{\mathcal{P}}}

\newtheorem{theorem}{Theorem}[section]
\newtheorem{proposition}[theorem]{Proposition}
\newtheorem{definition}[theorem]{Definition}
\newtheorem{lemma}[theorem]{Lemma}
\newtheorem{corollary}[theorem]{Corollary}
\newtheorem{remark}[theorem]{Remark}
\newtheorem{conjecture}[theorem]{Conjecture}

\title{Higher algebraic structures in \\Hamiltonian Floer theory}
\author{Oliver Fabert}
\thanks{O. Fabert, VU Amsterdam, The Netherlands. Email: oliver.fabert@gmail.com}
\begin{document}

\maketitle

\begin{abstract}
In this paper we show how the rich algebraic formalism of Eliashberg-Givental-Hofer's symplectic field theory (SFT) can be used to define higher algebraic structures in Hamiltonian Floer theory. Using the SFT of Hamiltonian mapping tori we show how to define a homotopy extension of the well-known Lie bracket and discuss how it can be used to prove the existence of multiple closed Reeb orbits. Furthermore we show how to define the analogue of rational Gromov-Witten theory in the Hamiltonian Floer theory of open symplectic manifolds. More precisely, we introduce a so-called cohomology F-manifold structure in Hamiltonian Floer theory and prove that it generalizes the well-known Frobenius manifold structure in rational Gromov-Witten theory. 
\end{abstract}

\tableofcontents
\markboth{O. Fabert}{Higher algebraic structures in Hamiltonian Floer theory} 

\section*{Introduction and summary}
The Floer theory of Hamiltonian symplectomorphisms is an important tool in symplectic geometry. Floer homology was invented by A. Floer to prove the Arnold conjecture about the number of symplectic fixed points and since then was improved to answer many other questions in symplectic geometry. Following M. Schwarz and P. Seidel (\cite{Sch}), there exists the so-called pair-of-pants product in Floer homology. Apart from the Arnold conjecture for degenerate Hamiltonians, it is used in proofs of the Conley conjecture and plays a crucial role in the definition of symplectic quasimorphisms. The applications of the pair-of-pants product build on its relation with the small quantum product of the underlying symplectic manifold. Following Piunikhin-Salamon-Schwarz (\cite{PSS}) there exists an ring isomorphism between Floer homology with its pair-of-pants product and the small quantum homology ring. On the other hand, the small quantum product only involves a very small part of rational Gromov-Witten theory, since it just counts holomorphic spheres with three marked points. In order to use the geometric information of all rational Gromov-Witten invariants, it is known, see \cite{MDSa}, that there also exists a big version of the quantum cup product, which recovers the full rational Gromov-Witten potential and is the main ingredient to equip the quantum homology with the structure of a Frobenius manifold. \\
    
Since it is the goal of this paper to show how the big quantum product and the corresponding Frobenius manifold structure generalize to Hamiltonian Floer theory, we quickly review the algebraic formalism of big quantum homology. First, in order to be able to work with a simpler Novikov field and to be able to prove transversality for all occuring moduli spaces, we assume thoughout this paper that $(M,\omega)$ is semimonotone in the sense that $\omega(A)=\tau\cdot c_1(A)$ for all $A\in\pi_2(M)$ with some fixed $\tau\geq 0$, see \cite{MDSa}. Note that this includes the case of monotone symplectic manifolds as well as all exact symplectic manifolds. Following (\cite{MDSa}, section 11.1) we let $\Lambda$ denote the universal Novikov field of all formal power series in the formal variable $t$ of even degree with rational coefficients $n_{\eps}\in\IQ$, $$\Lambda\,\ni\,\lambda=\sum_{\epsilon\in\IR} n_{\eps} t^{\eps}:\,\#\{\eps\leq c:\;n_{\eps}\neq 0\}<\infty\,\textrm{for all}\,c\in\IQ.$$ With this we define the quantum homology of $M$ to be its singular homology with coefficients in the Novikov field, $\QH^*(M)=\H^*(M)\otimes\Lambda$. Note that here and throughout the entire paper, every grading is to be understood in the $\IZ_2$-sense. \\

The big difference between singular and quantum homology lies in the fact that there is a quantum product on $\QH^*(M)$ which is defined by counting holomorphic spheres $$u:S^2\to M,\,\,\CR_J u=du+J(u)\cdot du\cdot i=0$$ in the symplectic manifold $M$ equipped with a $\omega$-compatible almost complex structure $J$. On the other hand, while the quantum cup product on $\QH^*(M)$ just involves moduli spaces of holomorphic spheres with three marked points, the rational Gromov-Witten potential $F$ of $(M,\omega)$ also takes into account spheres with more than three marked points. The moduli space $\IM_{r+1}(A)$ of holomorphic spheres  with $r+1$ additional marked points consists of tuples $(u,z_0,\ldots,z_{r-1},z_{\infty})$, where $u$ denotes a holomorphic sphere and $z_0,\ldots,z_{r-1},z_{\infty}$ are marked points on $S^2$. We assume that $(z_0,z_1,z_{\infty})=(0,1,\infty)$ which ensures that there are no nontrivial automorphisms of the sphere $S^2\backslash\{z_0,z_1,\ldots,z_{r-1},z_{\infty}\}$. Using the evaluation map $\ev=(\ev_0,\ldots,\ev_{r-1},\ev_{\infty}): \IM_{r+1}(A)\to M^{r+1}$ given by 
$$\ev(u,z_2,\ldots,z_{r-1}) \;=\; (u(0),u(1),u(z_2),\ldots,u(z_{r-1}),u(\infty)),$$ the Gromov-Witten potential of $(M,\omega)$ is defined as the generating function \\$F=F(q)$, $q=(q_1,\ldots,q_K)$ given by $$\sum_r \frac{1}{(r+1)!} \sum_{\alpha_0,\ldots,\alpha_{\infty}} \int_{\IM_{r+1}(A)} \bigwedge_{i=0}^{r-1}\ev_i^*\theta_{\alpha_i}\wedge\ev_{\infty}^*\theta_{\alpha_{\infty}}\cdot\; q_{\alpha_0}\cdot\ldots\cdot q_{\alpha_{\infty}}t^{c_1(A)}.$$ Here $(q_1,\ldots,q_K)$ are formal variables assigned to the basis of homology classes $\theta_1,\ldots,\theta_K\in H^*(M)$ with the parity $|q_{\alpha}|=|\theta_{\alpha}|\in\IZ_2$. Note that they can be viewed as coordinates of a linear space $\operatorname{Q}$ which is canonically isomorphic to $\QH^*(M)$ by identifying $\theta_{\alpha}\in\QH^*(M)$ with the unit vector $e_{\alpha}=(0,\ldots,1,\ldots,0)\in\operatorname{Q}$. \\

Employing gluing of holomorphic spheres one can show that the Gromov-Witten potential satisfies the so-called WDVV-equations, which in turn can be interpreted as associativity equation for a family of new products which now involves moduli spaces of holomorphic spheres with an arbitrary number of additional marked points and hence contains the full information of the  rational Gromov-Witten theory of $(M,\omega)$. The idea is to use triple derivatives of the Gromov-Witten potential to define a product $\star_q: T_q\operatorname{Q}\otimes T_q\operatorname{Q}\to T_q\operatorname{Q}$ on the tangent space at each $q\in \operatorname{Q}$ by
$$\frac{\del}{\del q_{\alpha_0}}\star_q\frac{\del}{\del q_{\alpha_1}} \;=\; \sum_{\alpha_{\infty},\beta} \eta_{\alpha_{\infty},\beta} \cdot \Bigl(\frac{\del^3 F}{\del q_{\alpha_0}\del q_{\alpha_1}\del q_{\beta}}\Bigr)(q)\cdot\; \frac{\del}{\del q_{\alpha_{\infty}}},$$ where $\eta_{\alpha_{\infty},\beta}$ denotes the Poincare pairing. Here observe that the tangent space $T_q\operatorname{Q}$ at each $q=(q_1,\ldots,q_K)$ is canonically isomorphic to the original space $\QH^*(M)$ by identifying $\del/\del q_{\alpha}$ with $\theta_{\alpha}$. The coefficient $(\del^3 F/\del q_{\alpha_0}\del q_{\alpha_1}\del q_{\beta})(q)$ is given by $$\sum_r \frac{1}{(r-2)!}\sum_{\alpha_2,\ldots,\alpha_{r-1},A} \int_{\IM_{r+1}(A)}\bigwedge_{i=0}^{r-1}\ev_i^*\theta_{\alpha_i}\wedge\ev_{\infty}^*\theta_{\beta}\cdot\; q_{\alpha_2}\cdot\ldots\cdot q_{\alpha_{r-1}}t^{c_1(A)}.$$ 

The new product is called the \emph{big quantum cup product} and is indeed a deformation of the quantum product $\star_0$ in the sense that we indeed have $\star_q=\star_0$ at $q=(q_1,\ldots,q_K)=0$ as $$\Bigl(\frac{\del^3 F}{\del q_{\alpha_0}\del q_{\alpha_1}\del q_{\beta}}\Bigr)(0)\;=\; \sum_A\int_{\IM_3(A)}\ev_0^*\theta_{\alpha_0}\wedge\ev_1^*\theta_{\alpha_1}\wedge\ev_{\infty}^*\theta_{\beta}\cdot\; t^{c_1(A)}.$$ Note that here we view the small quantum product as a product on the tangent space to $\Q$ at $q=0$, $$\star=\star_0: T_0\operatorname{Q}\otimes T_0\operatorname{Q}\to T_0\operatorname{Q}.$$ Algebraically the big quantum (cup) product can equivalently be viewed as product on the space of vector fields $\T^{(1,0)}\Q$ on $\Q$, $$\star: \T^{(1,0)}\Q\otimes \T^{(1,0)}\Q\to \T^{(1,0)}\Q,$$ that is, it is an element in the space $\T^{(1,2)}\Q$ of $(1,2)$-tensor fields on $\Q$. Note that by an $(r,s)$-tensor field on $\Q$ we mean a map which assigns to every $q\in\Q$ a $\Lambda$-linear map from the $s$-fold tensor product of $T_q\Q\cong\Q$ to its $r$-fold tensor product. We emphasize that the big quantum product is the key ingredient of the Frobenius manifold structure on $\operatorname{Q}\cong\QH^*(M)$ defined by Dubrovin, see \cite{MDSa}, \cite{DZ}.\\

In this paper we show how the big quantum product and the corresponding Frobenius manifold structure generalize to the Floer theory of a Hamiltonian symplectomorphism $\phi$, extending the relation between the small quantum product and the pair-of-pants product. We emphasize that, in contrast to the bulk deformations considered in \cite{FOOO}, we deform the pair-of-pants product by actually introducing additional periodic orbits and not just additional marked points. As we show with an example at the end of the paper, this leads to new applications in the case of open symplectic manifolds; \emph{indeed we emphasize that the case of symplectic manifolds with contact-type boundary is explicitly considered in this paper.} In contrast to the classical bulk deformations we show that the arising moduli spaces have new codimension-one boundary components which leads to an enriched algebraic formalism. In particular, in the case of symplectic manifolds with contact-type boundary, we emphasize that it is in general not possible to deform the pair-of-pants product using any linear combination of periodic orbits when at least one of the periodic orbits of the Hamiltonian corresponds to a closed Reeb orbit on the boundary. Indeed we show that this is only possible if it is a solution to the Maurer-Cartan equation for the $L_{\infty}$-structure in Hamiltonian Floer theory that we introduce first. In order to elegantly solve all appearing compactness problems, we show that the corresponding new structures in Hamiltonian Floer theory naturally emerge in the rich geometric and algebraic structures of the rational symplectic field theory of the mapping torus $M_{\phi}$ of $\phi$ (\cite{EGH},\cite{F1}). \\

Starting first with the case of closed symplectic manifolds $M$, as a first result we give a proof of the folk theorem relating the cylindrical contact homology $\HC^{\cyl}_*(M_{\phi})$ with the Floer homologies of powers of $\phi$.   

\begin{proposition} 
For every $k\in\IN$ the natural identification between the chain subspace $C^k_*\subset C_*$ of $\HC^{\cyl}_*(M_{\phi})$ generated by the $k$-periodic good Reeb orbits and the subspace of $\IZ_k$-invariant elements in the chain space $\CF_*(\phi^k)$  of $\HF_*(\phi^k)$ given by $$C^k_*\to\CF_*(\phi^k)^{\IZ_k},\;\gamma\mapsto\frac{1}{\kappa_{\gamma}}(x\pm\ldots\pm\phi^{k-1}(x))$$ is compatible with the boundary operators in cylindrical contact homology and Floer homology. Here $\kappa_{\gamma}$ denotes the multiplicity of $\gamma$ and the sign is determined by the behaviour of the orientation of the moduli spaces in cylindrical contact homology under the action of rotating the asymptotic markers. Together with the fact that $\HF_*(\phi^k)^{\IZ_k}=\HF_*(\phi^k)$ for a Hamiltonian symplectomorphism $\phi$, it follows that the cylindrical contact homology of the mapping torus $M_{\phi}$ is naturally isomorphic to the sum of the Floer homologies of all powers of $\phi$, $$\HC^{\cyl}_*(M_{\phi}) \;\cong\; \bigoplus_k \HF_*(\phi^k).$$ \end{proposition}

For the well-definedness of cylindrical contact homology one crucially uses the compactness result for holomorphic curves in cylindrical manifolds established in \cite{BEHWZ}. While this requires $M_{\phi}$ and hence $M$ to be closed, it is well-known that the Hamiltonian Floer homology groups can still be defined in the case of symplectic manifolds with contact-type boundary. For this one has to assume that Hamiltonian function extends to the completion in such a way that  it grow linearly of constant slope $m>0$ with the $\IR$-coordinate $\theta$ on the cylindrical end of $M$. Using the above relation between cylindrical contact homology and Hamiltonian Floer homology, we immediately get the cylindrical contact homology $\HC^{\cyl}_*(M_{\phi})$ is still well-defined and an invariant of the symplectic manifold, after fixing the asymptotic linear slope of the Hamiltonian in the cylindrical end, it is an invariant of the symplectic manifold. \\

After embedding Hamiltonian Floer theory into the framework of the symplectic field theory, we want to illustrate how the higher algebraic structures defined in \cite{EGH} lead to higher algebraic structures in Hamiltonian Floer theory.  \\

Following \cite{EGH} and \cite{F1}, the full contact homology of $M_{\phi}$ is defined as the homology of a chain complex where the differential is defined by counting unparametrized punctured holomorphic curves with one positive cylindrical end but an arbitrary number of negative cylindrical ends. In the same way as the cylindrical contact homology has an immediate interpretation in Hamiltonian Floer theory, the same is indeed true for the new algebraic structures arising from the more general moduli spaces. Instead of defining the chain complex of full contact homology, it is already sketched in \cite{EGH} that one alternatively one can use  the information of all moduli spaces to introduce a sequence of bracket-type operations on cylindrical contact homology and hence in Hamiltonian Floer theory. In the same way as the definition of cylindrical contact homology can be extended to the case of symplectic manifolds with contact-type boundaries, we show that the same holds for these operations. 

\begin{theorem}
By counting holomorphic curves with multiple cylindrical ends one can define an $L_{\infty}$-structure on the cylindrical contact homology and hence in Hamiltonian Floer theory. It is still well-defined when one passes from closed symplectic manifolds to symplectic manifolds with contact-type boundary. Furthermore, after fixing the asymptotic linear slope of the Hamiltonian in the cylindrical end, it is an invariant of the symplectic manifold up to homotopy.
\end{theorem}

We then show that the our $L_{\infty}$-structure indeed extends the well-known Lie bracket in Hamiltonian Floer theory, see \cite{A} and \cite{R}. For this we show 

\begin{proposition} The coefficients appearing in the definition of the $L_{\infty}$-structure count Floer solutions $u:\Si\to M$ in the sense (\cite{R}, 6.1) with one positive puncture with varying conformal structure and simultaneously rotating asymptotic markers. In particular, the above $L_{\infty}$-structure extends the Lie bracket in Hamiltonian Floer theory defined in (\cite{A}, 2.5.1) \end{proposition}

While it immediately follows that the pair-of-pants product defines a product on the cylindrical contact homology of $M_{\phi}$, $$\star_0: \HC^{\cyl}_*(M_{\phi}) \otimes \HC^{\cyl}_*(M_{\phi}) \to \HC^{\cyl}_*(M_{\phi}),$$ we define in this paper a \emph{big version of the pair-of-pants product} which now, in analogy with the big quantum product from rational Gromov-Witten theory, is supposed to be a product on vector fields. \\

As in the definition of big quantum homology, we start by introducing formal variables $q_{\gamma}$ for each closed Reeb orbit $\gamma$, the chain space $C_*=\bigoplus_k \CF_*(\phi^k)$ of the cylindrical contact homology can be identified with the tangent space $T_0\tilde{\Q}$ at zero of an infinite-dimensional linear coordinate space $\tilde{\Q}$ by identifying $\gamma\in C_*$ with $\del/\del q_{\gamma}\in T_0\tilde{\Q}$. While the (small) pair-of-pants product can be defined as a map $\star_0: T_0\tilde{\Q}\otimes T_0\tilde{\Q}\to T_0\tilde{\Q}$, the big version of the pair-of-pants product is supposed to provide us with a family of products $\star_q: T_q\tilde{\Q}\otimes T_q\tilde{\Q}\to T_q\tilde{\Q}$ on each tangent space, that is, a (1,2)-tensor field $\star\in\T^{(1,2)}\tilde{\Q}$. Since we require that the big pair-of-pants product in Hamiltonian Floer theory generalizes the big quantum product on $\QH^*(M)$, the vector field product $\star$ shall generalize the (small) pair-of-pants product in the same way as the big quantum product from rational Gromov-Witten theory generalizes the small quantum product on quantum homology.  \\

\begin{definition} On the chain level, the \emph{big pair-of-pants product} is defined to be the $(1,2)$-tensor field $\star\in\T^{(1,2)}\tilde{\Q}$ on the super space $\tilde{\Q}$ given by $$\sum_{\gamma^+,\gamma_0,\gamma_1}\Bigl(\sum_{\Gamma,A} \frac{1}{(r-2)!} \frac{1}{\kappa_{\gamma_0}\kappa_{\gamma_1}\kappa^{\Gamma}} \cdot\frac{1}{k}\cdot\#\IM^{\gamma^+}_{\gamma_0,\gamma_1}(\Gamma)\cdot \;t^{c_1(A)}q^{\Gamma}\Bigr)dq_{\gamma_0}\otimes dq_{\gamma_1}\otimes\frac{\del}{\del q_{\gamma^+}}$$ where $k=k_0+\ldots+k_{r-1}$ is the period of the closed orbit $\gamma^+$. \end{definition}

The new moduli spaces $\IM^{\gamma^+}_{\gamma_0,\gamma_1}(\Gamma)$ used to define the big pair-of-pants product can be identified with moduli spaces of Floer solutions $u:\Si\to M$ in the sense of \cite{R} with an arbitrary number of negative punctures and \emph{varying conformal structure and fixed asymptotic markers.} Note that they can be considered as an intermediate case between the moduli spaces used to define the $L_{\infty}$-structure and the moduli spaces of Floer solutions considered in \cite{R} and \cite{Sch} used to define the TQFT structure in Floer theory. \\

For the definition of the new moduli spaces we use that the fact that, in contrast the well-known case of contact manifolds, there exists a natural projection from $M_{\phi}\cong S^1\times M$ to the circle. Instead of fixing the asymptotic markers, we equivalently fix the induced branched covering map $h=(h_1,h_2):\Si\to\RS$. For this we introduce an additional marked point $z^*$ on the underlying punctured Riemann sphere $\Si=S^2\backslash\{z_0,\ldots,z_{r-1},\infty\}$, which we require to get mapped to $0\in S^1$. Of course, it still remains to constrain the additional marked point a priori without using the map. It is the key idea for our definition of the big pair-of-pants product that we can use the first two (special) negative punctures as well as the positive puncture to fix unique coordinates on $\Si$ by setting $z_0=0$, $z_1=1$.  \\

It is a consequence of a compactness problem for the moduli spaces of holomorphic curves that the big pair-of-pants product $\star$ does not descend to a well-defined vector field product on Floer homology directly. Indeed, since we allow the conformal structure to vary, in the codimension-one boundary of compactification of the moduli spaces there will not only appear an extra cylinder; instead the holomorphic curves typically break into two curves with multiple cylindrical ends. In analogy to the fact that the (small) pair-of-pants product satisfies properties like commutativity and associativity only after passing to cylindrical homology, we prove that the same true for the big pair-of-pants product. Since the big pair-of-pants product counts holomorphic curves with an arbitrary number of cylindrical ends, it is the key observation that, by passing from the small to the big pair-of-pants product, the cylindrical contact homology differential needs to be replaced by the differential of full contact homology. \\  

In order to formulate the main theorem, we observe that the boundary operator of full contact homology, that has been used to define the $L_{\infty}$-structure before, defines a vector field $\tilde{X}\in\T^{(1,0)}\tilde{\Q}$ on the super space $\tilde{\Q}$, $$\tilde{X}=\sum_{\gamma}\Bigl(\sum_{\Gamma,A} \frac{1}{r!}\frac{1}{\kappa^{\Gamma}}\sharp\IM^{\gamma}(\Gamma,A)/\IR\cdot\; q^{\Gamma} t^{c_1(A)}\Bigr)\frac{\del}{\del q_{\gamma}}\;\in\; \T^{(1,0)}\tilde{\Q}.$$ Furthermore, observe that, by using homotopy transfer, the tensor fields $\tilde{X}\in\T^{(1,0)}\tilde{\Q}$, $\tilde{\star}\in\T^{(1,2)}\tilde{\Q}$ on the chain level $\tilde{\Q}=\bigoplus_k \CF_*(\phi^k)$ define unique (up to homotopy) tensor fields $X\in\T^{(1,0)}\Q$, $\star\in\T^{(1,2)}\Q$ on homology $\Q=\bigoplus_k \HF_*(\phi^k)$.\\

\begin{theorem}
Together with the boundary operator from full contact homology, the big pair-of-pants product equips the sum of the Floer homology groups $\bigoplus_k \HF_*(\phi^k)$ with the structure of a cohomology F-manifold in such a way that, at the tangent space at zero, we recover the (small) pair-of-pants product on Floer homology. 
\end{theorem}

Cohomology F-manifolds are generalizations of Frobenius manifolds defined by Merkulov in \cite{M1}, \cite{M2}. Among other things, the vector field product $\star$ now just lives on a differential graded manifold instead of a graded linear space and one drops the requirement for an underlying potential. In other words, a cohomology F-manifold is a differential graded manifold $\Q_X:=(\Q,X)$ equipped with a graded commutative and associative product for vector fields $$\star: \T^{(1,0)}\Q_X\otimes\T^{(1,0)}\Q_X\to\T^{(1,0)}\Q_X.$$ In particular, since the vector field $X$ contains the same information as the $L_{\infty}$-structure that we have introduced before, we see that the algebraic formalism for the big pair-of-pants product heavily builds on the latter. \\

Following \cite{CS} and \cite{K}\footnote{In \cite{K} they are called (formal pointed) $Q$-manifolds}, a differential graded manifold is given by a pair of a graded linear space $\Q$ (more generally, a formal pointed graded manifold $\Q$) and a cohomological vector field $X$ on $\Q$. With the latter we mean a vector field $X\in \T^{(1,0)}\Q$ which satisfies $[X,X]=2X^2=0$ and $X(0)=0$. After lifting all structures by homotopy transfer from the chain space to homology, in our theorem the underlying graded linear space is cylindrical contact homology, i.e., the sum of the Floer cohomologies, $\Q=\HC_*^{\cyl}(M_{\phi})=\bigoplus_k \HF_*(\phi^k)$, while the cohomological vector field $X$ indeed encodes the $L_{\infty}$-structure that we have defined before. The crucial property of differential graded manifolds is that they have well-defined space of functions $\T^{(0,0)}\Q_X=H_*(\T^{(0,0)}\Q,X)$ and vector fields $\T^{(1,0)}\Q_X=H_*(\T^{(1,0)}\Q,[X,\cdot])$ (in \cite{M2} they are called the homology structure sheaf and the homology tangent sheaf, respectively), which in turn allows us to define arbitrary spaces of tensor fields $\T^{(r,s)}\Q_X$. \\
  
By extending the isomorphism proof for contact homology from \cite{EGH} and \cite{F1}, we show that, for different choices of auxiliary data like almost complex structures and Hamiltonian perturbations, the resulting cohomology F-manifolds are isomorphic in a canonical way. More precisely, we show

\begin{theorem} Fixing the asymptotic linear slope of the Hamiltonians in the cylindrical end, it follows that for different choices of auxiliary data like $S^1$-dependent Hamiltonian functions $H^{\pm}$ and $\omega$-compatible almost complex structures $J^{\pm}$, the resulting cohomology F-manifolds are isomorphic and we hence obtain a new invariant of the symplectic manifold with contact-type boundary. In the case when the symplectic manifold $M$ is closed, we recover the cohomology F-manifold structure on $\bigoplus_{k\in\IN}\QH^*(M)$ given by the big quantum product. \end{theorem}

After introducing these new structures in symplectic homology, we turn to applications and links to mirror symmetry. \\

Concerning applications, we show that the nontriviality of the $L_{\infty}$-structure in Hamiltonian Floer theory can be used to prove the existence of closed Reeb orbits. Let $\Delta$ denote the BV operator on Hamiltonian Floer homology as defined in \cite{A},\cite{R}. 

\begin{theorem}
If the $L_{\infty}$-structure on $\HC^{\cyl}_*(M_{\phi})=\bigoplus_k\HF_*(\phi^k)$  is not trivial, then there is at least one closed Reeb orbit on the contact-type boundary of $M$. If $M$ is Liouville and the $L_{\infty}$-structure is not trivial on $\Ker\Delta$, then there either exist two simple closed Reeb orbits or one homologically trivial Reeb orbit on the contact-type boundary of $M$.
\end{theorem}

In particular, it follows that the $L_{\infty}$-structure is trivial in the case when $M$ is closed. \\

Finally we outline how our newly defined algebraic structures on Hamiltonian Floer homology can be used to rigorously formulate a conjecture of Seidel in \cite{Se} on the relation between the quantum cohomology of the quintic three-fold and the Hamiltonian Floer homology of a divisor complement in it. \\

This paper is dedicated to the memory of my friend Alex Koenen who died in an hiking accident shortly before the first version of this paper was finished.    

\section{Cylindrical contact homology and Floer homology}

\noindent\emph{Floer theory for symplectomorphisms}\\

Let $(M,\omega)$ be a symplectic manifold and let $H:S^1\times M\to\IR$ be a time-dependent Hamiltonian. The resulting Hamiltonian symplectomorphism is the time-one map $\phi=\phi^1_H$ of the flow of the time-dependent symplectic gradient $X^H_t$ of $H_t=H(t,\cdot)$. In order to be able to prove transversality for all occuring moduli spaces, see the generalization of the results from \cite{F1} in the appendix, as well as to be able to work with a simpler Novikov field, we assume that $(M,\omega)$ is semimonotone in the sense that $\omega(A)=\tau\cdot c_1(A)$ for all $A\in\pi_2(M)$ with some fixed $\tau\geq 0$, see \cite{MDSa}. Note that this includes the case of monotone symplectic manifolds as well as all exact symplectic manifolds. Furthermore we assume that, after choosing Hamiltonian perturbations as in the appendix, all fixed points of the Hamiltonian symplectomorphism $\phi$ are nondegenerate, in particular, isolated. We first briefly review the definition of the Floer homology groups $\HF_*(\phi)$ of the Hamiltonian symplectomorphism $\phi=\phi^1_H$. For this we assume, until mentioned otherwise, that the symplectic manifold is closed. \\

Let $\P(\phi)$ denote the set of one-periodic orbits of the flow of $X^H_t$. Using the evaluation at $0\in S^1$, note that the one-periodic orbits $x: S^1\to M$ are in one-to-one correspondence with fixed points $p=\phi(p)$ of the Hamiltonian symplectomorphism $\phi$ via evaluation at the base point $0\in S^1$, $p=x(0)$. Unambiguously we will not distinguish between one-periodic orbits and the corresponding fixed point. Using the Conley-Zehnder index $\CZ(x)$ of $x$, we can view $x$ as a $\IZ_2$-graded object with grading $|x|=\CZ(x)$. For the definition of the Conley-Zehnder index, we follow \cite{EGH} and choose circles representing a basis of $H_1(M)$. Then one can choose for every one-periodic orbit $x$ a oriented spanning surface with boundary given by $x$ and a linear combination of the aforementioned circles. After fixing unitary trivializations of $TM$ along the circles, note that the spanning surface can be used to define a unique unitary trivialization of the pullback bundle $x^*TM$. While the latter clearly depends on the choice of the spanning surface, the parity of the Conley-Zehnder index is independent of this choice. Following (\cite{MDSa}, section 11.1) we let $\Lambda$ denote the universal Novikov ring of all formal power series in the formal variable $t$ of even degree with rational coefficients $n_{\eps}\in\IQ$, $$\Lambda\,\ni\,\lambda=\sum_{\epsilon\in\IR} n_{\eps} t^{\eps}:\,\#\{\eps\leq c:\;n_{\eps}\neq 0\}<\infty\,\textrm{for all}\,c\in\IQ.$$ Note that our choice of coefficients ensures that $\Lambda$ is indeed a field. With this we introduce the Floer chain groups $\CF_*(\phi)$ to be the $\IZ_2$-graded vector space spanned by all fixed points $x\in \P(\phi)$ with coefficients in the field $\Lambda$. \\

In order to define the boundary operator $\del: \CF_*(\phi)\to \CF^{*+1}(\phi)$, we start with choosing an $\omega$-compatible almost complex structure $J$ on $M$. Note that for the necessary regularity result in Floer homology it is sufficient to work with fixed $J$ as long as one is allowed to perturb the $S^1$-dependent Hamiltonian function; for details see the appendix. For two given fixed points $x^-,x^+\in \P(\phi)$, let $\IM_{x^-}^{x^+}(A)$ denote the moduli space of cylinders $u:\IR\times S^1\to M$ satisfying Floer's perturbed Cauchy-Riemann equation $$\CR_{J,H} u = \del_s u + J(u) \cdot (\del_t u - X^H_t(u)) = 0,$$ connecting the corresponding two one-periodic orbits in the sense that $u(s,t)\to x^{\pm}(t)$ as $s\to\pm\infty$ and representing the absolute homology class $A\in H_2(M)$. Furthermore it is important to observe that there is a natural $\IR$-action on this space and we assume that elements in $\IM_{x^-}^{x^+}$ are equivalence classes under this $\IR$-action. Note that for the latter we use that for the definition of the Conley-Zehnder index we have already chosen a spanning disk for every contractible closed orbit $x$. With this we define the boundary operator $\del: \CF_*(\phi)\to \CF^{*+1}(\phi)$ as $$\del x^- \;=\; \sum_{x^+,A} \# \IM_{x^-}^{x^+}(A) \cdot\; x^+ t^{c_1(A)},$$ where $\# \IM_{x^-}^{x^+}(A)$ denotes the algebraic count of elements in the moduli space of cylinders modulo $\IR$-shift in the case when $\ind(u)=1$ and is equal to zero otherwise. \\

In order to ensure that we always get a finite count, we use that $\IM^{x^+}_{x^-}(A)$ is compact when $\ind(u)=1$. On the other hand, when $\ind(u)=2$, $\IM^{x^+}_{x^-}(A)$ can be compactified to a one-dimensional moduli space with boundary given by $$\del^1 \IM^{x^+}_{x^-}(A) = \bigcup \IM^{x^+}_{x}(A^+) \times \IM^{x}_{x^-}(A^-), $$ where the union runs over all fixed points $x\in\P(\phi)$ with $\ind(u^+)=\ind(u^-)=1$ for $(u^+,u^-)\in \IM^{x^+}_{x}(A^+) \times \IM^{x}_{x^-}(A^-)$ and $A^++A^-=A$. Translating the above compactness result into algebra, we have shown that we indeed have $\del \circ \del = 0$, so that we can define the Floer homology groups as $$\HF_*(\phi) = H_*(\CF_*(\phi),\del).$$ Furthermore it can be shown that the homology groups for different choices of almost complex structures and Hamiltonian symplectomorphisms $\phi$ are isomorphic. In particular, when $\phi$ is Hamiltonian, then the Floer homology groups $\HF_*(\phi)$ are isomorphic to the quantum homology groups $\QH_*(M)$, which here are defined as the singular homology groups of $M$ with coefficients in the universal Novikov ring $\Lambda$ from above. \\

\noindent\emph{Cylindrical contact homology}\\

Now we review how Floer homology can be embedded into the framework of symplectic field theory, by giving a rigorous proof of a folk theorem. As above we first assume that the symplectic manifold $M$ is closed. We start with the observation that (parametrized) one-periodic Hamiltonian orbits $x: S^1\to M$ are in one-to-one correspondence with unparametrized one-periodic orbits $\gamma$ of the canonical vector field $\del_t$ on the corresponding mapping torus $M_{\phi}=\IR\times M/\{(t,p)\sim (t+1,\phi(x))\}$ by setting $\gamma: S^1\to M_{\phi}$, $\gamma(t)=(t,x)$ where $x$ is viewed as the corresponding fixed point. Following \cite{BEHWZ}, see also \cite{F1}, note that $M_{\phi}$ naturally carries a stable Hamiltonian structure in the sense of \cite{BEHWZ} given by $(\tilde{\omega}=\omega,\tilde{\lambda}=dt)$  with Reeb vector field $\tilde{R}=\del_t$. As described in \cite{F1}, the stable Hamiltonian manifold $M_{\phi}$ can be identified with $S^1\times M$ equipped with the $H$-dependent stable Hamiltonian structure $(\tilde{\omega}^H=\omega+dH_t\wedge dt,\tilde{\lambda}^H=dt)$ with Reeb vector field $\tilde{R}^H=\del_t+X^H_t$, where the underlying diffeomorphism between $M_{\phi}$ and $S^1\times M$ is given by the Hamiltonian flow, $S^1\times M\to M_{\phi}$, $(t,p)\mapsto (t,\phi^t_H(p))$. \\

Generalizing the one-to-one correspondence between (parametrized) orbits $x^{\pm}$ in $M$ and unparametrized orbits $\gamma^{\pm}$ in $M_{\phi}\cong S^1\times M$, one can show that the moduli space of (parametrized) Floer cylinders $\IM_{x^-}^{x^+}$ connecting $x^+$ and $x^-$ can be identified with the moduli space of unparametrized $\tilde{J}$-holomorphic cylinders in $\IR\times M_{\phi}\cong \RS\times M$ converging to $\{+\infty\}\times \gamma^+$ and $\{-\infty\}\times \gamma^-$ in the cylindrical ends. For this observe that the $\omega$-compatible almost complex structure $J$ on $(M,\omega)$ and the $S^1$-dependent Hamiltonian $H_t$ naturally defines a cylindrical almost complex structure $\tilde{J}=\tilde{J}^H$ on $\RS\times M$ in the sense of \cite{BEHWZ}, compatible with the stable Hamiltonian structure, by setting $\tilde{J}\del_s = \del_t+X^H_t$ and requiring that $\tilde{J}$ agrees with $J$ on $TM$, see \cite{F1} and \cite{BEHWZ}. Then an easy computation shows, see also (\cite{F1}, proposition 2.2 and 2.4), that unparametrized $\tilde{J}$-holomorphic maps $\tilde{u}:\IR\times S^1\to\IR\times S^1\times M$ with $\tilde{u}(s,t)\to(\pm\infty,\gamma^{\pm}(t))$ as $s\to\pm\infty$ are in one-to-one correspondence with connecting Floer cylinders $u\in\IM_{x^-}^{x^+}$. For this observe that $\tilde{u}$ can be written as a tuple $\tilde{u}=(h,u)$, where $u$ satisfies Floer's perturbed Cauchy-Riemann equation and $h$ is an automorphism of the cylinder which, after applying the inverse automorphism, we can always assume to be the identity. Note that the natural $\IR$-action on $\IM_{x^-}^{x^+}(A)$ corresponds to the natural $\IR$-symmetry on the space of $\tilde{J}$-holomorphic maps to the cylindrical almost complex manifold $\IR\times M_{\phi}$. \\

Following \cite{EGH}, the cylindrical contact homology $\HC^{\cyl}_*=\HC^{\cyl}_*(M_{\phi})$ of the mapping torus $M_{\phi}$ is the homology of a chain complex, $\HC^{\cyl}_*=H_*(C_*,\del)$, where the chain space $C_*$ is now defined to be the linear space generated by the closed unparametrized good orbits $\gamma$ of the Reeb vector field with coefficients in the universal Novikov ring $\Lambda$ from before, where for the definition of good orbits we refer to the discussion below. Note that the period of each closed orbit in $M_{\phi}\cong S^1\times M$ agrees with the degree of the map to the base circle and hence the chain space naturally splits, $C_*=\bigoplus_k C^k_*$, where $C^k_*$ is generated by the orbits of period $k\in\IN$. As before we work with a $\IZ_2$-grading given by $|\gamma|=\CZ(\gamma)$, where $\CZ(\gamma)$ denotes the Conley-Zehnder index for closed Reeb orbits defined in \cite{EGH}. The boundary operator $\del: C_*\to C_*$ is defined as $$\del \gamma^-\;=\;  \frac{1}{\kappa_{\gamma^-}}\cdot \sum_{\gamma^+,A}\# \IM_{\gamma^-}^{\gamma^+}(A) \cdot\; \gamma^+ t^{c_1(A)},$$ where $\kappa_{\gamma}$ denotes the multiplicity of the closed orbit $\gamma$, see \cite{EGH}. \\

Here $\IM^{\gamma^+}_{\gamma^-}=\IM^{\gamma^+}_{\gamma^-}(A)$ denotes the moduli space of unparametrized $\tilde{J}$-holomorphic cylinders $\tilde{u}:\IR\times S^1\to\IR\times M_{\phi}$ converging to $\gamma^+$ and $\gamma^-$ near the cylindrical ends, $\tilde{u}(s,t+\tau^{\pm})\to (\pm\infty,\gamma^{\pm}(kt))$ as $s\to\pm\infty$ for some $\tau^{\pm}\in S^1$ . For the latter observe that, although we now want to consider the orbits as unparametrized objects, in the original definition from \cite{EGH} one arbitrarily fixes a parametrization by choosing a special point on each closed Reeb orbit $\gamma$. Note that, as in the definition of the moduli spaces $\IM_{x^-}^{x^+}$ in Floer homology, there is a natural $\IR$-action on this space and we assume that elements in $\IM_{\gamma^-}^{\gamma^+}$ are equivalence classes under this $\IR$-action. \\

\noindent\emph{Cylindrical contact homology and Floer homology}\\

In order to provide a natural link between the chain complexes of cylindrical contact homology and Floer homology,  we will modify the original definition and choose on each $k$-periodic orbit not one but $k$ special points naturally given by the intersection of the orbit with the fibre over $\pi^{-1}(0)\subset M_{\phi}$ of the projection $M_{\phi}\to S^1$. In order to cure for the resulting overcounting, we assume that every special point comes with the rational weight $1/k$. Note that when $\gamma$ is multiply-covered then some of these special points might coincide and we sum the weights correspondingly; in particular, when $\gamma$ is a $k$-fold cover of a one-periodic orbit, then we agree with the original definition in \cite{EGH}. The $k$ special points in turn define $k$ asymptotic markers (directions) at each cylindrical end and we follow \cite{EGH} and assume that the moduli spaces $\IM^{\gamma^+}_{\gamma^-}$ are made up of $\tilde{J}$-holomorphic maps $\tilde{u}$ as above together with asymptotic markers at each cylindrical end up to reparametrization of the underlying cylinder. In contrast to the original definition in \cite{EGH}, note that our choices of special points on $\gamma^+$ (and $\gamma^-$) lead to a natural $\IZ_k(\times\IZ_k)$-action on $\IM^{\gamma^+}_{\gamma^-}$ and we assume that every unparametrized holomorphic cylinder with asymptotic markers in $\IM^{\gamma^+}_{\gamma^-}$ comes equipped with the weight given by the product of the weights assigned to the special points defining the asymptotic markers. \\ 

While the closed orbits of period one are in bijection with the fixed points $x$ in $\P(\phi)$, note that for general $k\in\IN$ the fixed points in $\P(\phi^k)$ are in $k$-to-one-correspondence with closed orbits of period $k$ when the underlying orbit is simple. For this observe that for every fixed point $x\in\P(\phi^k)$ the points $\phi(x),\ldots,\phi^{k-1}(x)$ are also fixed points of $\phi^k$ which induces a natural $\IZ_k$-action on the chain space $\CF_*(\phi^k)$.  Note that the Conley-Zehnder indices of $\phi^i(x)$ agree for all $i=0,\ldots,k-1$ by symmetry reasons, since the corresponding one-periodic orbits just differ by reparametrization and the spanning surface $u$ for $x$ naturally defines spanning surfaces for all $\phi^i(x)$. On the other hand, $x,\phi(x),\ldots,\phi^{k-1}(x)$ all represent the same unparametrized $k$-periodic Reeb orbit $\gamma$. Furthermore the Conley-Zehnder index of $\gamma$ defined in \cite{EGH} agrees with the Conley-Zehnder index of $x$ if we use the same unitary trivialization to define the index for $\gamma$.  More precisely, there is indeed a one-to-one correspondence between the $k$ fixed points and the $k$ special points that we have chosen on $\gamma$ above. While it is not hard to see from our discussion above that the Floer chain complex for $\phi$ is contained in the chain complex of the cylindrical contact homology of $M_{\phi}$, we now show that the full cylindrical contact homology has an interpretation in terms of the Floer homologies of all powers $\phi^k$  of the underlying Hamiltonian symplectomorphism $\phi=\phi^1_H$ (defined using the same $\omega$-compatible almost complex structure). Let $\CF_*(\phi^k)^{\IZ_k}\subset\CF_*(\phi^k)$ denote the subspace of $\IZ_k$-invariant elements. By symmetry reason it follows that the boundary operator restricts to a boundary operator $\del: \CF_*(\phi^k)^{\IZ_k}\to\CF^{*+1}(\phi^k)^{\IZ_k}$.

\begin{proposition}\label{cylindrical}
For every $k\in\IN$ the natural identification between the chain subspace $C^k_*\subset C_*$ generated by the $k$-periodic good Reeb orbits and the subspace of $\IZ_k$-invariant elements in $\CF_*(\phi^k)$ given by $$C^k_*\to\CF_*(\phi^k)^{\IZ_k},\;\gamma\mapsto\frac{1}{\kappa_{\gamma}}(x\pm\ldots\pm\phi^{k-1}(x))$$ is compatible with the boundary operators in cylindrical contact homology and Floer homology. Here $\kappa_{\gamma}$ denotes the multiplicity of $\gamma$ and the sign is determined by the behaviour of the orientation of the moduli spaces in cylindrical contact homology under the action of rotating the asymptotic markers. Together with the fact that $\HF_*(\phi^k)^{\IZ_k}=\HF_*(\phi^k)$ for a Hamiltonian symplectomorphism $\phi$, it follows that the cylindrical contact homology of the mapping torus $M_{\phi}$ is naturally isomorphic to the sum of the Floer homologies of all powers of $\phi$, $$\HC^{\cyl}_*(M_{\phi}) \;\cong\; \bigoplus_k \HF_*(\phi^k).$$ \end{proposition}

\begin{proof} 
The proof for $k=1$ is already given above, since have shown that connecting Floer cylinders in $\IM^{x^+}_{x^-}$ are in one-to-one correspondence with unparametrized cylinders in $\IM^{\gamma^+}_{\gamma^-}$. For the case when $k$ is an arbitary natural number, observe first that the moduli space $\IM^{\gamma^+}_{\gamma^-}$ is only non-empty when $\gamma^+$ and $\gamma^-$ have the same period $k$ by homological reasons. It again follows from (\cite{F1},proposition 2.2), see also (\cite{F1}, proposition 2.4), that $\tilde{u}=(h,u):\RS\to\IR\times M_{\phi}\cong\RS\times M$ is a $\tilde{J}$-holomorphic cylinder precisely when $h:\RS\to\RS$ is holomorphic and $u:\RS\to M$ satisfies the Floer equation $\CR_{J,H,h}(u)=\Lambda^{0,1}_J(du+X^H_{h_2}\otimes dh_2)=0$. When $\tilde{u}$ represents an element in $\IM^{\gamma^+}_{\gamma^-}$ with $k$-periodic orbits, then it follows that $h$ is a $k$-fold unbranched covering map from the cylinder to itself, which in turn implies that $u$ satisfies the Floer equation for the pair $(J,H^k)$ with the $1/k$-periodic Hamiltonian $H^k_t=kH_{kt}$. After applying an automorphism of the domain, note that for every $\tilde{u}\in\IM^{\gamma^+}_{\gamma^-}$ we can always assume that the induced covering map $h:\RS\to\RS$ is given by $h(s,t)=(ks,kt)$. \\

After fixing the $k$-fold covering map $h$ using the action of the automorphism group, note that there still remains a $\IZ_k$-action. In analogy to the relation between closed orbits and fixed points, it follows that there is a $k$-to-one correspondence between unparametrized $\tilde{J}$-holomorphic cylinders $\tilde{u}$ and cylinders $u:\RS\to M$ satisfying the Floer equation for $(J,H^k)$ given by reparametrization of the underlying cylinder. In particular, we have $$\#\IM^{\gamma^+}_{\gamma^-} \;=\; \frac{1}{k} \sum_{i^{\pm}=0}^{k-1} \pm\#\IM^{\phi^{i^+}(x^+)}_{\phi^{i^-}(x^-)}$$ in case that $x^{\pm},\ldots,\phi^{k-1}(x^{\pm})$ represents the orbit $\gamma^{\pm}$, where the sign results from the behaviour of the orientation of the moduli space $\IM^{\gamma^+}_{\gamma^-}$ defined in (\cite{BM}, section $3$) under the action of rotating the asymptotic markers, see (\cite{BM}, section $5$). For this observe that, just as one needs to fix an orientation on the unstable manifold of every critical point in order to orient the moduli spaces of gradient flow lines in Morse homology, here one needs to make a similar choice for every periodic orbit in order to be able to orient all occuring moduli spaces of cylinders in a coherent way. Note that when $\gamma^+$ or $\gamma^-$ is multiply-covered with multiplicity $\kappa_{\gamma^{\pm}}$ and hence some of the fixed points $x^{\pm},\ldots,\phi^{k-1}(x^{\pm})$ agree, we still need to count them as different, since for each holomorphic cylinder in $\IM^{\gamma^+}_{\gamma^-}$ there are now $\kappa_{\gamma^{\pm}}$ possible directions for the asymptotic marker. On the other hand, in the same way as the closed one-periodic orbits $x_0,\ldots,x_{k-1}:S^1\to M$ corresponding to fixed points $x^{\pm},\ldots,\phi^{k-1}(x^{\pm})$ are obtained by reparametrization, $x_i(t)=x(t+i/k)$,  the moduli space $\IM^{x^+}_{x^-}$ from Floer homology is naturally isomorphic to the moduli space $\IM^{\phi^i(x^+)}_{\phi^i(x^-)}$ for all $0\leq i\leq k-1$ via reparametrization. It follows that $$\sum_{i^-=0}^{k-1}\pm\#\IM^{x^+}_{\phi^{i^-}(x^-)}\;=\;\sum_{i^-=0}^{k-1}\pm\#\IM^{\phi^{i^+}(x^+)}_{\phi^{i^-}(x^-)}$$ for all $0\leq i^+\leq k-1$. Together with the above identity we find that $$\#\IM^{\gamma^+}_{\gamma^-} \;=\; \sum_{i^-=0}^{k-1} \pm\#\IM^{x^+}_{\phi^{i^-}(x^-)}.$$  

Using this we can show that the chain map $C^k_*\to\oplus_k \CF_*(\phi^k)^{\IZ_k}$, $\gamma\mapsto \frac{1}{\kappa_{\gamma}}(x\pm\ldots\pm\phi^{k-1}(x))$ has the desired property. Note in particular that, when $\gamma$ is a bad orbit in the sense of \cite{EGH}, then the alternating sum on the right-hand-side gives zero, for more details see again (\cite{BM}, section $5$). It then follows that with respect to the above identification the differential $\del: \oplus_k \CF_*(\phi^k)^{\IZ_k}\to\oplus_k \CF_*(\phi^k)^{\IZ_k}$ in Floer homology agrees with the differential in cylindrical contact homology, 
\begin{eqnarray*} 
&&\del\Big(\frac{1}{\kappa_{\gamma^-}}(x^-\pm\ldots\pm\phi^{k-1}(x^-))\Big) \\&&=\; \frac{1}{\kappa_{\gamma^-}}\cdot \sum_{x^+,A} \Big(\sum_{i=0}^{k-1} \pm\#\IM^{x^+}_{\phi^i(x^-)}(A)\Big)\cdot\; x^+ t^{c_1(A)} \\&&=\; \frac{1}{\kappa_{\gamma^-}}\cdot \sum_{x^+,A} \pm\#\IM^{\gamma^+}_{\gamma^-}(A)\cdot\, x^+ t^{c_1(A)} \\&&=\; \frac{1}{\kappa_{\gamma^-}}\cdot \sum_{\gamma^+,A} \#\IM^{\gamma^+}_{\gamma^-}(A)\cdot\,\frac{1}{\kappa_{\gamma^+}}(x^+\pm\ldots\pm\phi^{k-1}(x^+)) t^{c_1(A)}. 
\end{eqnarray*} 
Finally, for Hamiltonian symplectomorphisms with sufficiently $C^2$-small Hamiltonian (depending on $k$) note that all fixed points of $\phi^k$ correspond to critical points of the underlying Hamiltonian and hence are already fixed points of $\phi$. It follows that $\HF_*(\phi^k)^{\IZ_k}=\HF_*(\phi^k)=\QH_*(M)$ for such small Hamiltonian symplectomorphisms. From the invariance properties of Floer homology we then get $\HF_*(\phi^k)^{\IZ_k}\cong\QH_*(M)\cong\HF_*(\phi^k)$ and hence $\HF_*(\phi^k)^{\IZ_k}=\HF_*(\phi^k)$ for all Hamiltonian symplectomorphisms. \end{proof}    

\noindent\emph{Generalization to symplectic manifolds with contact-type boundary}\\

After establishing the relation between cylindrical contact homology and Hamiltonian Floer homology for closed symplectic manifolds $M$, we now turn to the case of symplectic manifolds with contact-type boundary. A symplectic manifold is said to have contact-type boundary if in the neighborhood of the boundary there exists a vector field $Z$, called Liouville vector field, which is pointing outward and is transverse to the boundary $\del M$ and satisfies $\mathcal{L}_Z\omega=\omega$, where $\mathcal{L}$ denotes the Lie derivative. Defining the one-form $\lambda$ on $\del M$ by $\lambda=\iota(Z)(\omega)$, an easy exercise shows that $\lambda$ is a contact form in the sense that $\lambda\wedge(d\lambda)^{n-1}\neq 0$ with $2n=\dim M$. \\

Every symplectic manifold with contact-type boundary possesses a so-called completion $M\cup \IR^+\times \del M$, where $M$ and the cylindrical end $\IR^+\times \del M$ are glued along $(\{0\}\times)\del M$ by requiring that the Liouville vector field $Z$ agrees with the $\IR$-direction $\del_s$ on $\IR^+\times M$. In this paper we will not distinguish between the manifold with boundary and its completition, which we will also sometimes call it an open symplectic manifold with cylindrical end. A special example is a Liouville manifold which is a tuple $(M,\lambda)$ of an open manifold $M$ and a one-form $\lambda$ on $M$ such that $(M,\omega=d\lambda)$ is an exact symplectic manifold with a cylindrical end, where the globally defined vector field $Z$ is determined by $\lambda=\iota(Z)(\omega)$. Finally, in order to be able to study holomorphic curves in $M$, we assume that the completed manifold is equipped with a $\omega$-compatible almost complex structure $J$ which is cylindrical in $\IR^+\times M$ in the sense of \cite{BEHWZ}, in particular, $JZ$ is tangent to $\del M$ in the cylindrical end. \\

For the well-definedness of cylindrical contact homology one crucially uses the compactness result for holomorphic curves in cylindrical manifolds established in \cite{BEHWZ} which itself crucially relies on the fact that $M_{\phi}$ and hence $M$ is closed. On the other hand, it is well-known, see \cite{R}, that Hamiltonian Floer homology can be generalized to open symplectic manifolds with cylindrical ends, after requiring that the Hamiltonian function $H:S^1\times M\to\IR$ is of a special form on the cylindrical end. To this end, we follow \cite{R} and assume from now on that the underlying Hamiltonian function $H: S^1\times M\to\IR$ has \emph{asymptotic linear growth} of some slope $m>0$ in the sense that there exists some $\theta_0>0$ such that, when $x=(\theta,p)\in\IR^+\times\del M$ is a point in the cylindrical end, we have $H(t,\theta,p)=m\cdot \theta$ for $\theta\geq\theta_0$. \\

Since the $C^0$-bound for the Floer trajectory $u$ from (\cite{R}, lemma 2.1) immediately gives a $C^0$-bound for the map $\tilde{u}: \IR\times S^1\to \IR\times M_{\phi}$ ensuring that its image stays in a compact subset of the open mapping torus $M_{\phi}$, it follows that the compactness result for holomorphic curves from \cite{BEHWZ} still hold and hence the cylindrical contact homology of the mapping torus $M_{\phi}$ is still well-defined. \\

\begin{proposition} 
For a symplectic manifold with contact-type boundary and a Hamiltonian symplectomorphism $\phi$ of the special form described above, that is, where the underlying Hamiltonian has asymptotic linear growth, the cylindrical contact homology $\HC^{\cyl}_*(M_{\phi})$ is still well-defined and given by the sum of the Floer homologies $\HF_*(\phi^k)$. Fixing the linear slope of the Hamiltonian in the cylindrical end, it is an invariant of the symplectic manifold with boundary. \end{proposition}

\begin{proof}
It just remains to be shown that for two different choices of $\omega$-compatible almost complex structures $J^+$ and $J^-$ and Hamiltonian functions $H^+$ and $H^-$ \emph{with the same asymptotic slope $m>0$} the resulting cylindrical contact homologies $\HC^{\cyl}_*(M_{\phi^+})$ and $\HC^{\cyl}_*(M_{\phi^-})$ are still isomorphic. \\

For this recall from \cite{EGH} that there is a continuation map $\varphi: \HC^{\cyl}_*(M_{\phi^+})\to \HC^{\cyl}_*(M_{\phi^-})$ which is defined by counting holomorphic cylinders in $\RS\times M$ equipped with a $\IR$-dependent almost complex structure $\hat{J}$. It is explicitly determined by $\hat{J}=J_s$ on $TM$ and $\hat{J}\del_s=\del_t+X^H_{s,t}$. Here $J_s$ and $H_{s,t}$ are $\IR$-dependent families of $\omega$-compatible almost complex structures and $S^1$-dependent Hamiltonians which interpolate between the tuples $(J^+,H^+_t)$ and $(J^-,H^-_t)$ in the sense that $J_s=J^+$, $H_{s,t}=H^+_t$ for $s>s_0$ and $J_s=J^-$, $H_{s,t}=H^-_t$ for $s<-s_0$ (for some fixed $s_0>1$). Denoting $J^k_s=J_{ks}$, $H^k_{s,t}:=kH_{ks,kt}$, one can show as before that each (unparametrized) $\hat{J}$-holomorphic cylinder $\hat{u}$ in $\RS\times M$ connecting two Reeb orbits can be identified with a cylinder $u:\RS\to M$ satisfying the $\IR$-dependent Floer equation $\CR_{J^k,H^k}(u)=\Lambda^{0,1}_{J^k}(du+X^{H^k}_{s,t}\otimes dt)=0$ for some $k\in\IN$, see (\cite{F1}, theorem 5.2). \\

It follows that each of the natural continuation maps $\varphi$ corresponds to a family of continuation maps on the Floer homologies $\varphi^k: \HF_*((\phi^+)^k)\to \HF_*((\phi^-)^k)$ for all $k\in\IN$. On the other hand, assuming the $\IR$-dependent interpolating Hamiltonian $H_{s,t}:M\to\IR$, $(s,t)\in\RS$ is again chosen to have linear slope with respect to $\theta\in\IR$ in the cylindrical end, we can again employ the $C^0$-bounds from \cite{R} to show that the relevant compactness results for holomorphic curves in cobordisms from \cite{BEHWZ} still hold. \\ 

In order to see that the $\IZ_k$-invariant part  of the Floer homology group of $\phi^k$ agrees with $\HF_*(\phi^k)$ as in the case of Hamiltonian symplectomorphisms on closed symplectic manifolds, we follow \cite{BO} and assume that the $S^1$-dependent Hamiltonians $H^k_t$ defining $\phi^k$ for each $k\in\IN$ are obtained as a small perturbation of a $S^1$-independent Hamiltonian $H=H_0: M\to\IR$ using a small Morse function on each closed orbit of $H$. Identifying each $k$-periodic orbit of $H$ with the circle using an appropriate parametrization, we assume that each Morse function is $1/k$-periodic to obtain the required $\IZ_k$-symmetry on the set of generators of the chain complex underlying $\HF_*(\phi^k)$. On the other hand, it can be directly seen from the definition of the Morse-Bott differential in \cite{BO} that all perturbed orbits represent the same generator on homology. \end{proof}

\section{$L_{\infty}$-structure in Hamiltonian Floer theory}

After embedding Hamiltonian Floer homology into the framework of the symplectic field theory of Hamiltonian mapping tori, we want to illustrate how the higher algebraic structures defined in \cite{EGH} lead to higher algebraic structures in Hamiltonian Floer theory. We again start with the case where the symplectic manifold $(M,\omega)$ is closed. \\

\noindent\emph{Contact homology}\\

In order to use the algebraic formalism of \cite{EGH}, let us first introduce for every good closed Reeb orbit $\gamma$ in $M_{\phi}$ a formal variable $q_{\gamma}$ with the same $\IZ_2$-grading, $|q_{\gamma}|=|\gamma|$. Note that they can be viewed as coordinates of an abstract graded linear space $\tilde{\Q}$ over $\Lambda$ which is canonically isomorphic to the chain space of cylindrical contact homology of $M_{\phi}$ by identifying the generator $\gamma$ in the chain space with the unit vector $e_{\gamma}$ in $\tilde{\Q}$. Alternatively, observing that the tangent space $T_0\tilde{\Q}\cong\tilde{\Q}$ of $\tilde{\Q}$ (at $0\in\tilde{\Q}$) is spanned by the vectors $\del/\del q_{\gamma}$ with $|\del/\del q_{\gamma}|=|q_{\gamma}|=|\gamma|$, we can equivalenty identify $\gamma$ with $\del/\del q_{\gamma}$.  In any case, we will freely jump between both pictures. \\

Following \cite{EGH} and \cite{F1}, the \emph{full} contact homology of $M_{\phi}$ is defined as the homology of the chain complex $$\HC_*(M_{\phi})=H_*(\T^{(0,0)}\tilde{\Q},\tilde{X}).$$ Here the chain space $\T^{(0,0)}\tilde{\Q}$ denotes the algebra of polynomial functions (=$(0,0)$-tensor fields) on the chain space $\tilde{\Q}$ with values in $\Lambda$, which alternatively (as done in \cite{EGH}) can be described as the graded commutative algebra spanned by the formal variables $q_{\gamma}$ over the Novikov ring $\Lambda$. It follows that, as a graded linear space over $\Lambda$, it is spanned by monomials $q^{\Gamma}:=q_{\gamma_1}\cdot\ldots\cdot q_{\gamma_{\ell}}$ for all finite collections of closed Reeb orbits $\Gamma=(\gamma_1,\ldots,\gamma_{\ell})$.\\

On the other hand, the space $\T^{(1,0)}\tilde{\Q}$ of vector fields on $\tilde{\Q}$ is spanned, as a linear space over $\Lambda$, by formal products of the form $q^{\Gamma}\cdot \del/\del q_{\gamma^+}$. Note that every element in $\T^{(1,0)}\tilde{\Q}$ defines a linear map from the space of functions $\T^{(0,0)}\tilde{\Q}$ into itself, by viewing it as a derivation satisfying a graded version of the Leibniz rule. This said, the boundary operator in full contact homology defined in \cite{EGH} can be encoded in the vector field $$\tilde{X}=\sum_{\gamma^+}\Bigl(\sum_{\Gamma,A} \frac{1}{r!}\frac{1}{\kappa^{\Gamma}}\sharp\IM^{\gamma^+}(\Gamma,A)\cdot\; q^{\Gamma} t^{c_1(A)}\Bigr)\frac{\del}{\del q_{\gamma^+}}\;\in\; \T^{(1,0)}\tilde{\Q}$$ on $\tilde{\Q}$, called \emph{cohomological vector field}, defined by counting unparametrized punctured $\tilde{J}$-holomorphic curves with one positive cylindrical end but an arbitrary number of negative cylindrical ends. \\

For every closed unparametrized orbit $\gamma^+$ (of period $k\in\IN$) and every ordered set of closed unparametrized orbits $\Gamma=(\gamma_0,\ldots,\gamma_{r-1})$ (of periods $k_0,\ldots,k_{r-1}$ with $k_0+\ldots+k_{r-1}=k$) of the Reeb vector field on $M_{\phi}\cong S^1\times M$, the moduli space $\IM^{\gamma^+}(\Gamma)=\IM^{\gamma^+}(\Gamma,A)$ consists of equivalence classes of tuples $(\tilde{u},z_0,\ldots,z_{r-1})$ together with an asymptotic marker (direction) at each $z_i$, where $(z_0,\ldots,z_{r-1})$ is a collection of marked points on $\IC= S^2\backslash\{\infty\}$ and $\tilde{u}=(h,u):\Si\to\IR\times M_{\phi}\cong\RS\times M$ is a $\tilde{J}$-holomorphic map from the resulting punctured sphere $\Si = \IC\backslash\{z_0\ldots,z_{r-1}\} = S^2\backslash\{z_0\ldots,z_{r-1},z_{\infty}=\infty\}$ to the cylindrical almost complex manifold $\IR\times M_{\phi}$. For the asymptotics we require that in compatible cylindrical coordinates $(s^+,t^+)$ near $z_{\infty}$ and $(s_i,t_i)$ near $z_i$ that $\tilde{u}(s^+,t^+)\to (+\infty,\gamma^+(kt^+))$ as $s^+\to +\infty$ and $\tilde{u}(s_i,t_i)\to (-\infty,\gamma_i(k_i t_i))$ as $s_i\to -\infty$ for all $i=0,\ldots,r-1$. \\

Note that the asymptotic markers (and hence the cylindrical coordinates) at each puncture are not fixed, but, as in the definition of cylindrical contact homology, they are fixed by special points on the closed Reeb orbits. For the latter observe that, although we want to consider the orbits as unparametrized objects, in the original definition from \cite{EGH} one arbitrarily fixes a parametrization by choosing a special point on each closed Reeb orbit $\gamma$. In order to provide a natural link between the chain complexes of cylindrical contact homology and Floer homology, note that we have modified the original definition and now choose on each $k$-periodic orbit not one but $k$ special points, naturally given by the intersection(s) of the orbit with the fibre over $\{0\}\times M\subset S^1\times M\cong M_{\phi}$. In order to cure for the resulting overcounting, we assume that every special point comes with the rational weight $1/k$. Note that when $\gamma$ is multiply-covered then some of these special points might coincide and we sum the weights correspondingly; in particular, when $\gamma$ is a $k$-fold cover of a one-periodic orbit, then we agree with the original definition in \cite{EGH}. The $k$ special points in turn define $k$ asymptotic markers (directions) at each cylindrical end. \\

As for cylindrical contact homology we consider unparametrized $\tilde{J}$-holomorphic curves and assume that elements in the moduli space $\IM^{\gamma^+}(\Gamma)$ are equivalence classes under the obvious action of the group of Moebius transformations on $\IC=S^2\backslash\{\infty\}$ and the natural $\IR$-shift on the cylindrical target manifold. Furthermore we assume, as before, that they are equipped with a rational weight given by the product of the rational weights of the special marked points defining the asymptotic markers. Note further that in the appendix we show how the results in \cite{F1} can be modified to achieve transversality for all moduli spaces $\IM^{\gamma^+}(\Gamma)$ using domain-dependent Hamiltonian perturbations. Finally it can be shown that to each $\tilde{J}$-holomorphic curve in $\IM^{\gamma^+}(\Gamma)$ one can still assign a class $A\in H_2(M)$. When $\Gamma$ consists of a single orbit $\gamma^-$, then we just get back the moduli spaces of cylindrical contact homology from before, $\IM^{\gamma^+}(\gamma^-)=\IM^{\gamma^+}_{\gamma^-}$. \\ 

\begin{center}
\includegraphics[height=6cm]{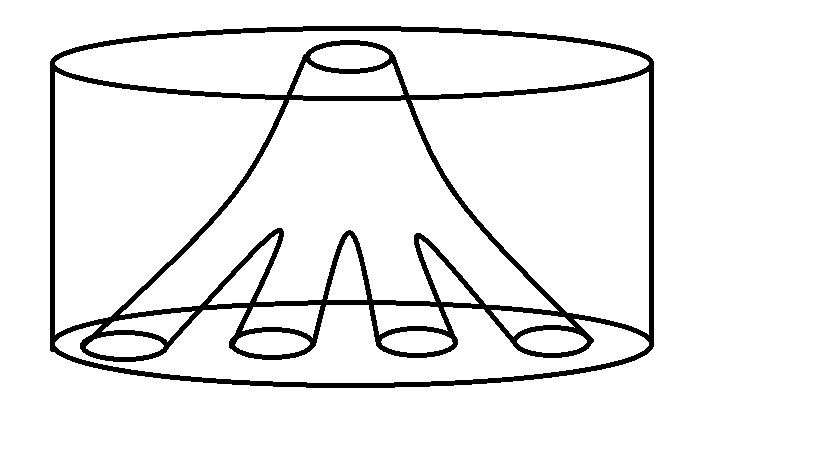}\\
\small{Punctured holomorphic curve in the cylindrical manifold $\IR\times M_{\phi}$}
\end{center}

It is shown in \cite{BEHWZ}, see also \cite{F1}, that the moduli space $\IM^{\gamma^+}(\Gamma)=\IM^{\gamma^+}(\Gamma,A)$ is compact when the index is one and, when the index is two, can be compactified to a one-dimensional moduli space with boundary $$\del^1 \IM^{\gamma^+}(\Gamma) = \bigcup \IM^{\gamma^+}(\Gamma') \times \IM^{\gamma}(\Gamma'')$$ formed by moduli spaces of the same type. For the latter we again use that $(M,\omega)$ is closed, see the proof of the next proposition. The above compactness result for one-dimensional moduli spaces translates into $\tilde{X}^2=0$, so that $\HC_*(M_{\phi})=H_*(\T^{(0,0)}\tilde{\Q},\tilde{X})$ is well-defined. \\

\noindent\emph{$L_{\infty}$-algebra}\\

In the same way as the cylindrical contact homology has an immediate interpretation in Hamiltonian Floer theory, the same is indeed true for the new algebraic structures arising from the more general moduli spaces $\IM^{\gamma^+}(\Gamma)$. For this we want to work for the moment with a different algebraic setup, see \cite{EGH} and \cite{CFL}. \\

Instead of using the information of all moduli spaces $\IM^{\gamma^+}(\Gamma)$ to define the chain complex of full contact homology, one can use it to define an $L_{\infty}$-structure on the cylindrical contact homology, see (\cite{M1}, section 2.4).

\begin{definition} A $L_{\infty}$-algebra structure on the graded vector space $\tilde{\Q}$ is a countable collection of multilinear maps $$\tilde{m_r}:\tilde{\Q}^r=\tilde{\Q}\times\ldots\times\tilde{\Q}\to\tilde{\Q}$$ which are homogeneous of degree $2-r$ and satisfy the higher Jacobi identities $$\sum_{k+\ell=r+1} \sum_{\sigma} (-1)^{k(\ell-1)}\epsilon\cdot \tilde{m}_{\ell}(\tilde{m}_k(v_{\sigma(1)},\ldots,v_{\sigma(k)}),v_{\sigma(k+1)},\ldots,v_{\sigma(n)})\,=\,0,$$ where the sign $\epsilon\in\{\pm 1\}$ is defined via the identity $v_{\sigma(1)}\wedge\ldots\wedge v_{\sigma(n)}=\epsilon\cdot v_1\wedge\ldots\wedge v_n$ and the second sum runs over the set of all permutations $\sigma:\{1,\ldots,n\}\to\{1,\ldots,n\}$ with $\sigma(1)<\ldots<\sigma(k)$ and $\sigma(k+1)<\ldots<\sigma(n)$. \end{definition}

Note that the first three higher Jacobi identities have the form 
\begin{eqnarray*} 
&r=1:& \del^2\,=\,0,\\
&r=2:& \del [v_1,v_2]\,=\,[\del v_1,v_2]+(-1)^{\bar{v}_1}[v_1,\del v_2],\\
&r=3:& [[v_1, v_2], v_3]+(-1)^{(\bar{v}_1+\bar{v}_2)\bar{v}_3}[[v_3, v_1], v_2]+(-1)^{\bar{v}_1(\bar{v}_2+\bar{v}_3)}[[v_2, v_3], v1]\\&&\,=\, \tilde{m}_3(\del v_1,v_2,v_3)+(-1)^{\bar{v}_1}\tilde{m}_3(v_1,\del v_2,v_3)+(-1)^{\bar{v}_1+\bar{v}_2}\tilde{m}_3(v_1,v_2,\del v_3),
\end{eqnarray*}
where $\del:=\tilde{m}_1$, $[\cdot,\cdot]=\tilde{m}_2$ and $\bar{v}\in\{\pm 1\}$ denotes the $\IZ_2$-grading of $v\in\tilde{\Q}$. \\

Defining a hierarchy of operations $$\tilde{m}_r:\tilde{\Q}^r=\tilde{\Q}\times\ldots\times\tilde{\Q}\to\tilde{\Q}$$ by $$\tilde{m}_r(\gamma_1,\ldots,\gamma_r):=\sum_{\gamma}\frac{1}{\kappa^{\Gamma}}\sharp\IM^{\gamma}(\Gamma)\cdot\;\gamma$$ for $\Gamma=(\gamma_1,\ldots,\gamma_r)$, it follows that we can polynomially expand $\tilde{X}$ as $$\tilde{X}(q)\;=\;\sum_r\frac{1}{r!} \sum_{\Gamma} q^{\Gamma}\cdot \tilde{m}_r(\gamma_1,\ldots,\gamma_r),$$ with $\Gamma=(\gamma_1,\ldots,\gamma_r)$ at all points $q=(q_{\gamma})\in\tilde{\Q}$. Note that here we identify each closed orbit $\gamma$ as usual with $\del/\del q_{\gamma}$. Then it follows from an easy computation, see \cite{K} or \cite{M1}, that 

\begin{proposition} The identity $\tilde{X}^2=0$ immediately shows that the hierarchy of operations $(\tilde{m}_r)$ satisfy the $L_{\infty}$-relations. \end{proposition}

Note that in \cite{CFL} the authors show that the full algebraic structure of SFT including holomorphic curves with arbitrary many positive ends and genus leads to the structure of an $\IBL_{\infty}$-structure, so the above proposition is already contained in their statement. Note that, since we consider mapping tori and regularity is proven without referring to polyfold theory, see also the appendix, our algebraic structures are rigorously defined. \\

In particular, the operation $m_1$ agrees with the boundary operator of cylindrical contact homology. Using homotopy transfer, it follows that we obtain a $L_{\infty}$-structure with a hierarchy of operations $m_r:\HC^{\cyl}_*(M_{\phi})\times\ldots\times\HC^{\cyl}_*(M_{\phi}) \to \HC^{\cyl}_*(M_{\phi})$ such that $m_2=[\cdot,\cdot]: \HC^{\cyl}_*(M_{\phi}) \times \HC^{\cyl}_*(M_{\phi}) \to \HC^{\cyl}_*(M_{\phi})$ is the Lie bracket already considered in \cite{CFL}. In the same way, note that we obtain a new vector field $X\in\T^{(1,0)}\Q$ on $\Q:=\HC^{\cyl}_*(M_{\phi})$ such that $\HC_*(M_{\phi})=H_*(\T^{(0,0)}\Q,X)$. We emphasize that, as on the chain level, the $L_{\infty}$-structure as well as the vector field $X$ contain the same algebraic information, so we switch freely between both formulations. Furthermore, when it is clear from the context whether we consider the $L_{\infty}$-structure on the chain or the homology level, we drop the tilde from the notation. \\     

In order to see that we have actually introduced a $L_{\infty}$-structure in Hamiltonian Floer theory, it suffices to recall that the cylindrical contact homology is given by the sum of Floer homologies, $\HC^{\cyl}_*(M_{\phi})=\bigoplus_k\HF_*(\phi^k)$. It follows that the $L_{\infty}$-operations $m_r: \HC^{\cyl}_*(M_{\phi})\times\ldots\times\HC^{\cyl}_*(M_{\phi}) \to \HC^{\cyl}_*(M_{\phi})$ are indeed given themselves by an infinite family of multilinear operations $$m_r^{k_1,\ldots,k_r}:\HF_*(\phi^{k_1})\times\ldots\times\HF_*(\phi^{k_r})\to\HF_*(\phi^k),$$ where $k=k_1+\ldots+k_r$. \\

\noindent\emph{Generalization to symplectic manifolds with contact-type boundary}\\

We have already seen in the last section that the cylindrical contact homology $\HC^{\cyl}_*(M_{\phi})$, which in \cite{EGH} is originally only defined for closed manifolds, can be generalized to the case where the underlying symplectic manifold has contact-type boundary. It is one of the key observations of this paper that the $L_{\infty}$-structure on it also generalizes to the case of open symplectic manifolds. This crucially relies on the following lemma, originally proven in \cite{F1}. \\

For this observe that every map from a punctured Riemann sphere $\Si$ to $\IR\times M_{\phi}\cong\RS\times M$ we can naturally written as a pair of maps $\tilde{u} = (h,u): \Si \to (\RS) \times M$.

\begin{lemma} The map $\tilde{u}:\Si\to\RS\times M$ is $\tilde{J}$-holomorphic precisely when $h=(h_1,h_2):\Si\to\RS$ is holomorphic and $u:\Si\to M$ satisfies the $h$-dependent perturbed Cauchy-Riemann equation of Floer type, 
\begin{eqnarray*} 
\CR_{J,H,h} u &=& \Lambda^{0,1}_J(du + X^H(h_2,u)\otimes dh_2) \\ 
              &=& du + X^H_{h_2}(u)\otimes dh_2 + J(u)\cdot(du+X^H_{h_2}(u)\otimes dh_2) \cdot i.
\end{eqnarray*}
\end{lemma}

\begin{proof} Observing that $\tilde{J}(t,p): T(\RS)\oplus TM\to T(\RS)\oplus TM$ is given by 
\begin{equation*} \tilde{J}(t,p)=\binom{\;\;i\;\;\;\;0\;\;}{\Delta(t,p)\;\;J(p)} \end{equation*} 
with $\Delta(t,p) = -X^H_t(p)\otimes ds + J(p) X^H_t(p)\otimes dt$ we compute 
\begin{eqnarray*} 
 && (dh,du) + \tilde{J}(h_2,u) \cdot (dh,du) \cdot i \\
 &=& (dh + i \cdot dh \cdot i, \\
 && \;du + (J(u)\cdot du - X^H(h_2,u)\otimes dh_1 + J(u)X^H_{h_2}(u)\otimes dh_2) \cdot i) \\
 &=& (\CR h, du - X^H_{h_2}(u)\otimes dh_1\cdot i + J(u)\cdot(du + X^H_{h_2}(u)\otimes dh_2) \cdot i). 
\end{eqnarray*} 
Finally observe that $dh_1 \cdot i = -dh_2$ if $\CR h =0$. \end{proof}

With this we can prove 

\begin{theorem}
For a symplectic manifold with contact-type boundary and a Hamiltonian symplectomorphism $\phi$ of the special form described above, that is, where the underlying Hamiltonian has asymptotic linear growth, the $L_{\infty}$-structure in Hamiltonian Floer theory is still well-defined. 
\end{theorem}

\begin{proof} While for the result about cylindrical contact homology we have used that every $\tilde{J}$-holomorphic cylinder $\tilde{u}$ is indeed given by a Floer trajectory $u: \IR\times S^1\to M$ and then used the well-known $C^0$-bound for Floer trajectories for our specially chosen Hamiltonian, here we proceed precisely along the same lines. Although we no longer consider cylinders, above we have shown that for maps $\tilde{u}=(h,u): \Si\to\IR\times M_{\phi}\cong (\IR\times S^1)\times M$ starting from arbitrary punctured spheres $\Si$ we have  $$\CR_{\tilde{J}}(\tilde{u})=0\;\Leftrightarrow\;\CR h=0\;\wedge\;\CR_{J,H}(u)=\Lambda^{0,1}_J(du + X^H_{h_2}(u)\otimes dh_2)=0,$$ where we again use the canonical diffeomorphism $M_{\phi}\cong S^1\times M$ given by the Hamiltonian flow. \\

In particular, the maps $u:\Si\to M$ still satisfy a Floer equation, where $h_2: \Si\to S^1$ is the second component of the branching map $h=(h_1,h_2)$ to the cylinder. More precisely, they are indeed \emph{Floer solutions in the sense of (\cite{R}, 2.5.1)}. That means that the perturbed Cauchy-Riemann operator $\CR_{J,H}$ from above belongs to the class of perturbed Cauchy-Riemann operators for punctured spheres for which A. Ritter proved a $C^0$-bound in (\cite{R}, lemma 19.1) to establish his TQFT structure on symplectic homology. Indeed we obviously have $d\beta\leq 0$ for $\beta=dh_2$. Although A. Ritter only considers moduli spaces of holomorphic curves with fixed conformal structure while we must allow the conformal structure to vary, his $C^0$-bound is sufficient since we still keep the asymptotic orbits fixed. Hence we still have that the punctured holomorphic curves in our moduli spaces stay in a compact subset of $M_{\phi}\cong S^1\times M$ and thus the required compactness results from \cite{BEHWZ} still apply. \end{proof}
 
\noindent\emph{$L_{\infty}$-morphisms}\\

In order to show that the $L_{\infty}$-structure in Hamiltonian Floer theory is indeed an invariant, it still remains to introduce morphisms. \\

Let $(\tilde{\Q}^+,\tilde{X}^+)$ and $(\tilde{\Q}^-,\tilde{X}^-)$ be pairs of chain spaces and cohomological vector fields for full contact homology, obtained using two different choices of cylindrical almost complex structures $\tilde{J}^{\pm}$ on $\RS\times M$ defined using two different choices of (domain-dependent) Hamiltonian functions $H^{\pm}$ and $\omega$-compatible almost complex structures $J^{\pm}$. As for the invariance proof for cylindrical contact homology we assume that we have chosen a smooth family $(H_s,J_s)$ of (domain-dependent) Hamiltonians and $\omega$-compatible almost complex structures interpolating between $(H^+,J^+)$ and $(H^-,J^-)$, so that we can equip the cylindrical manifold $\RS\times M$  with the structure of an almost complex manifold with cylindrical ends in the sense of \cite{BEHWZ}. We then count elements in moduli spaces $\widehat{\IM}^{\gamma^+}(\Gamma)=\widehat{\IM}^{\gamma^+}(\Gamma,A)$ of $\hat{J}$-holomorphic curves, which are defined analogous to the moduli spaces $\IM^{\gamma^+}(\Gamma)$, with the only difference that we no longer divide out the $\IR$-action in the target as the latter no longer exists. \\

It is shown in \cite{EGH}, see also \cite{F1}, that we can use these counts to define a chain map $\varphi^{(0,0)}:\T^{(0,0)}\tilde{\Q}^+\to\T^{(0,0)}\tilde{\Q}^-$ for the full contact homology by defining $$\varphi^{(0,0)}(q_{\gamma^+})\;=\;\sum_{\Gamma,A}\frac{1}{\kappa^{\Gamma}}\cdot\#\widehat{\IM}^{\gamma^+}(\Gamma,A)\cdot q^{\Gamma} t^{c_1(A)}$$ and $\varphi^{(0,0)}(q_{\gamma^+_1}\cdot\ldots\cdot q_{\gamma^+_s}):=\varphi^{(0,0)}(q_{\gamma^+_1})\cdot\ldots\cdot\varphi^{(0,0)}(q_{\gamma^+_s})$. On the other hand, it is shown in \cite{EGH} that the map is indeed compatible with the cohomological vector fields in the sense that $\tilde{X}^-\circ\varphi^{(0,0)}=\varphi^{(0,0)}\circ \tilde{X}^+$. After passing to homology, it can be shown, see \cite{EGH}, that the map $\varphi^{(0,0)}$ indeed defines an isomorphism of the full contact homology algebras,$$\varphi^{(0,0)}:\;H_*(\T^{(0,0)}\tilde{\Q}^+,\tilde{X}^+)\stackrel{\cong}{\longrightarrow}H_*(\T^{(0,0)}\tilde{\Q}^-,\tilde{X}^-).$$

Apart from the fact that $\varphi^{(0,0)}$ can be used to show that, for closed symplectic manifolds $M$, the full contact homology $\HC_*(M_{\phi})$ is independent of all auxiliary choices, it at the same time is a morphism of the corresponding $L_{\infty}$-structure, thus proving that the $L_{\infty}$-structure on $\HC^{\cyl}_*(M_{\phi})$ is well-defined up to homotopy. \\

Indeed, defining as above a hierarchy of operations $$\varphi_r: \tilde{\Q}^-\times\ldots\times\tilde{\Q}^-\to\tilde{\Q}^+$$ by $$\varphi_r(\gamma_1,\ldots,\gamma_r):=\sum_{\gamma}\frac{1}{\kappa^{\Gamma}}\sharp\widehat{\IM}^{\gamma}(\Gamma)\cdot\;\gamma$$ for $\Gamma=(\gamma_1,\ldots,\gamma_r)$, it follows that the chain map $\varphi^{(0,0)}:\T^{(0,0)}\tilde{\Q}^+\to\T^{(0,0)}\tilde{\Q}^-$ is equivalently described by requiring that $\varphi^{(0,0)}(f)=f\circ\varphi$ for all $f\in \T^{(0,0)}\tilde{\Q}^+$, where $$\varphi: \tilde{\Q}^-\to\tilde{\Q}^+,\;\varphi(q):=\sum_r\frac{1}{r!}\sum_{\Gamma} q^{\Gamma}\cdot \varphi_r(\gamma_1,\ldots,\gamma_r).$$ With this it is again an easy exercise, see \cite{K} or \cite{M1}, to prove that

\begin{proposition} The identity $\tilde{X}^-\circ\varphi^{(0,0)}=\varphi^{(0,0)}\circ \tilde{X}^+$ immediately shows that the hierarchy of operations $(\varphi_r)$ defines an $L_{\infty}$-morphism between the $L_{\infty}$-structures $(\tilde{m}_r^-)$ on $\tilde{\Q}^-$ and $(\tilde{m}_r^+)$ on $\tilde{\Q}^+$. \end{proposition}

We now want to generalize again to the case of symplectic manifolds with contact-type boundary. Here we prove the following

\begin{theorem}
Fixing the asymptotic linear slope of the Hamiltonian in the cylindrical end, the $L_{\infty}$-structure on $\HC^{\cyl}_*(M_{\phi})=\bigoplus_k \HF_*(\phi^k)$ is an invariant of the symplectic manifold with contact-type boundary, up to homotopy.
\end{theorem}

\begin{proof} Since the compactness results from \cite{BEHWZ} for symplectic cobordisms only apply to manifolds with cylindrical ends over closed stable Hamiltonian manifolds, we again need $C^0$-bounds for the holomorphic curves in the cobordism interpolating between the mapping tori $M_{\phi^+}, M_{\phi^-}\cong S^1\times M$. As before, see (\cite{F1}, theorem 5.2), it follows that these holomorphic maps from a punctured sphere $\Si$ are given by a branching map to the cylinder and a map $u:\Si\to M$ satisfying now the perturbed Cauchy-Riemann equation $\CR_{J,\tilde{H}}(u)=\Lambda^{0,1}_J(du + X^{\tilde{H}}_{h_1,h_2}(u)\otimes dh_2)$, where $\tilde{H}_{s,t}$ now interpolates between $H_t^+$ and $H_t^-$ from before. Assuming that $\tilde{H}$ grows asymptotically linear with fixed slope in the cylindrical end, independent of $s\in\IR$, the result in the appendix of (\cite{R}, lemma 19.1) can still be applied to give the desired $C^0$-bound and hence the required compactness statement. \end{proof}

We end this section by showing that the resulting $L_{\infty}$-structure indeed extends the well-known Lie bracket in Hamiltonian Floer theory as defined in \cite{A} and \cite{R}. For this we show 

\begin{proposition}\label{Lie} The coefficients $1/\kappa^{\Gamma}\cdot\#\IM^{\gamma^+}(\Gamma;A)$ appearing in the definition of the $L_{\infty}$-structure count Floer solutions $u:\Si\to M$ in the sense (\cite{R}, 6.1) with one positive puncture with varying conformal structure and simultaneously rotating asymptotic markers. In particular, the above $L_{\infty}$-structure extends the Lie bracket in Floer homology defined in (\cite{A}, 2.5.1) \end{proposition}

\begin{proof} As in the case of cylinders, following (\cite{F1}, proposition 2.2), the map $\tilde{u}=(h,u):\Si\to\IR\times M_{\phi}\cong\RS\times M$ is $\tilde{J}$-holomorphic precisely when $h:\Si\to\RS$ is holomorphic and $u:\Si\to M$ satisfies the Floer equation $\CR_{J,H,h}(u)=\Lambda^{0,1}_J(du+X^H_{h_2}\otimes dh_2)=0$. Forgetting the map $u$ and hence mapping $(\tilde{u},z_0,\ldots,z_{r-1})$ to $(h,z_0,\ldots,z_{r-1})$ defines a projection from $\IM^{\gamma^+}(\Gamma)$ to $\IM(k_0,\ldots,k_{r-1})$, where $\IM(k_0,\ldots,k_{r-1})=\IM^k(k_0,\ldots,k_{r-1})$ denotes the moduli space of holomorphic functions $h$ on $\IC$ with $r$ zeroes $z_1,\ldots,z_{r-1}$ of predescribed orders $k_0,\ldots,k_{r-1}$ up to Moebius transformations of $\IC$ and $\IR$-shift in the target, see the proof of (\cite{F1}, lemma 2.3). Since this forgetful map is further assumed to remember asymptotic markers at the punctures, one shall think of $\IM(k_0,\ldots,k_{r-1})$ as the moduli spaces of full contact homology in the case when $(M,\omega)$ is the point. Note that the fibre of this projection over each point in $\IM^k(k_0,\ldots,k_{r-1})$ is a moduli space of Floer maps $u:\Si\to M$ from a punctured Riemann surface of fixed conformal structure and with fixed asymptotic markers and hence precisely of the type as considered in (\cite{R},6.1); in particular, note that the special Floer equation $\CR_{J,H,h}(u)=0$ from above indeed satisfies the monotonicity assumption in \cite{R} since $\beta=dh_2$ immediately gives $d\beta\leq 0$. \\

On the other hand, further forgetting everything except the position of the punctures on the underlying punctured sphere, that is, mapping $(h,z_0,\ldots,z_{r-1})$ to $(z_0,\ldots,z_{r-1})$, defines a further natural forgetful map from $\IM^k(k_0,\ldots,k_{r-1})$ to the moduli space $\IM_{r+1}$ of conformal structures on the $r+1$-punctured sphere, and in order to finish the proof it remains to understand the fibre of this natural projection. Since the map $h$ exists for any choice of punctures $z_0,\ldots,z_{r-1}$ and orders $k_0,\ldots,k_{r-1}$ and is unique modulo a factor from $\IC^*\cong\RS$, observe we get an $S^1$-family of maps $h$ after dividing out the natural $\IR$-shift on the target manifold $\RS$. Furthermore, by varying $h=(h_1,h_2):\Si\to\RS$ inside this $S^1$-family, observe that the asymptotic markers at all punctures are indeed rotated simultaneously as they are mapped to the special point $0\in S^1$ by $h_2$. \\

While the map $h$ is indeed fixed (modulo $\IR$-shift) by fixing the asymptotic marker at any puncture, note that for a fixed choice of $h$ there is still a choice of $k$ asymptotic markers at the positive puncture and $k_0\cdot\ldots\cdot k_{r-1}$ $r$-tuples of asymptotic markers at the $r$ negative punctures, that is, the fibre of this second forgetful map is given by $S^1\times\IZ_k\times(\IZ_{k_0}\times\dots\times\IZ_{k_{r-1}})$. In order to explain the last factor $\IZ_{k_0}\times\dots\times\IZ_{k_{r-1}}$ as well as the appearance of the combinatorial factor $1/\kappa^{\Gamma}$, $\kappa^{\Gamma}=\kappa_{\gamma_0}\cdot\ldots\cdot\kappa_{\gamma_{r-1}}$ in front of $\#\IM^{\gamma^+}(\Gamma;A)$, note that for the identification of  cylindrical contact homology of $M_{\phi}$ with the sum of the Floer homologies for powers of $\phi$ we identify the closed Reeb orbit $\gamma_i$ with the weighted sum $1/\kappa_{\gamma_i}\cdot (x_i\pm\ldots\pm\phi^{k_i-1}(x_i))$ of fixed points of $\phi^{k_i}$ for all $i=0,\ldots,r-1$. On the other hand, since varying $h$ inside its $S^1$-family only rotates one of the $k$ asymptotic markers at the positive puncture onto the next one, it follows that $S^1\times\IZ_k$ precisely parametrizes all possible positions of the asymptotic marker at the positive puncture. \end{proof}

\begin{center}
\includegraphics[height=6cm]{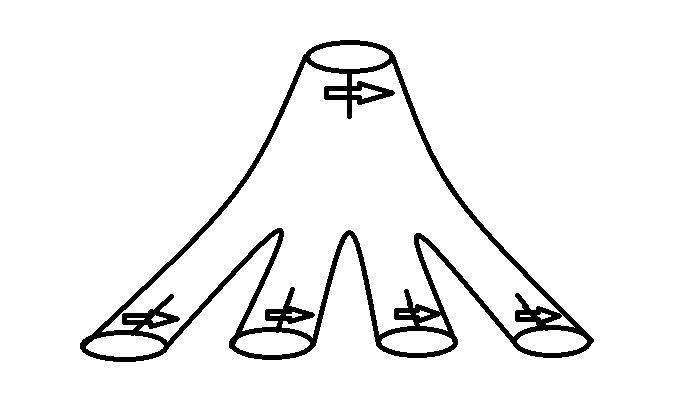}\\
\small{Punctured holomorphic curve with simultaneously rotating asymptotic markers}
\end{center}

\section{From $L_{\infty}$-algebras to differential graded manifolds}

\noindent\emph{Differential graded manifolds}\\

In the last section we have shown that the vector field $\tilde{X}\in\T^{(1,0)}\tilde{\Q}$ encodes an $L_{\infty}$-structure $(\tilde{m}_r)$ on $\tilde{\Q}$ (with $\tilde{m}_1=\del$ being the boundary operator of cylindrical contact homology) by polynomially expanding $\tilde{X}$. For the new algebraic structures arising from the definition of the big pair-of-pants product, which we will define in the next section, we need a more geometrical view of the $L_{\infty}$-structure. \\  

\begin{proposition} The pair $(\tilde{\Q}, \tilde{X})$ defines a (infinite-dimensional) differential graded manifold. \end{proposition} 

\begin{proof} For the definition of a (formal pointed) differential graded manifold we refer to \cite{CS}, see also (\cite{K}, subsection 4.3). First, it follows from the degree conventions in SFT that $\tilde{X}$ has odd degree. Using the definition of the graded Lie bracket, the master equation for the boundary operator in full contact homology further immediately implies that $$[\tilde{X},\tilde{X}]=2\tilde{X}^2=0.$$ Together with $\tilde{X}(0)=0$ this shows that $\tilde{X}\in\T^{(1,0)}\tilde{\Q}$ is indeed a \emph{cohomological vector field} in the sense of \cite{CS},\cite{K}. \end{proof}

For the translation between the languages of differential graded manifolds and of $L_{\infty}$-structures, we refer to \cite{K}, see also \cite{M1}. \\

As already mentioned above in the summary, the important property for us is that on the resulting differential graded manifold $\tilde{\Q}_{\tilde{X}}:=(\tilde{\Q},\tilde{X})$ one still has a space of functions $\T^{(0,0)}\tilde{\Q}_{\tilde{X}}$, called its homology structure sheaf in \cite{M2}, and a space of vector fields $\T^{(1,0)}\tilde{\Q}_{\tilde{X}}$, called its homology tangent sheaf in \cite{M2}, which in turn can be used to define arbitrary tensor fields $\T^{(r,s)}\tilde{\Q}_{\tilde{X}}$. Their elements can actually be viewed as functions and vector fields on some abstract singular space $\tilde{X}^{-1}(0)/\sim$ given by the solutions of the Maurer-Cartan equation of the equivalent $L_{\infty}$-structure modulo gauge; since we do not need this interpretation, we will not elaborate on this geometric picture any further here. \\
 
While $\T^{(0,0)}\tilde{\Q}_{\tilde{X}}$ agrees with contact homology, the space of vector fields $\T^{(1,0)}\tilde{\Q}_{\tilde{X}}$ is defined using the Lie bracket by  $$\T^{(1,0)}\tilde{\Q}_{\tilde{X}}\;:=\; H_*(\T^{(1,0)}\tilde{\Q},[\tilde{X},\cdot]).$$  On the other hand, we will use that the higher tensor fields $\T^{(r,s)}\tilde{\Q}_{\tilde{X}}$ can also be directly defined. For this we use the fact, which was already used in \cite{EGH} in order to define so-called satellites, that the above identity for $\tilde{X}$ implies that the induced Lie derivative $\LL_{\tilde{X}}$ defines a boundary operator on arbitrary tensor fields $\T^{(r,s)}\tilde{\Q}$,  $$\LL_{\tilde{X}}: \T^{(r,s)}\tilde{\Q}\to \T^{(r,s)}\tilde{\Q},\; \LL_{\tilde{X}}\circ\LL_{\tilde{X}}= \LL_{[\tilde{X},\tilde{X}]}=0.$$ Note that, in analogy to the definitions of $\T^{(0,0)}\tilde{\Q}$ and $\T^{(1,0)}\tilde{\Q}$, the space $\T^{(r,s)}\tilde{\Q}$ is spanned, as a graded linear space over $\Lambda$, by formal products of the form $$q^{\Gamma} \cdot dq_{\gamma^-_1}\otimes\ldots\otimes dq_{\gamma^-_s}\otimes\frac{\del}{\del q_{\gamma^+_1}}\otimes\ldots\otimes\frac{\del}{\del q_{\gamma^+_r}}.$$  On the other hand, since the Lie derivative commutes with the contraction of tensors, we find that indeed $$\T^{(r,s)}\tilde{\Q}_{\tilde{X}} = H_*(\T^{(r,s)}\tilde{\Q},\LL_{\tilde{X}}).$$\\

\noindent\emph{Morphisms}\\

Following (\cite{CS}, subsection 3.3) and \cite{K}, note that a morphism $\varphi$ between two differential graded manifolds $(\tilde{\Q}^-,\tilde{X}^-)$ and $(\tilde{\Q}^+,\tilde{X}^+)$ is a linear map $\varphi^{(0,0)}: \T^{(0,0)}\tilde{\Q}^+\to \T^{(0,0)}\tilde{\Q}^-$ which is compatible with the cohomological vector fields in the sense that $\tilde{X}^-\circ\varphi^{(0,0)}=\varphi^{(0,0)}\circ\tilde{X}^+$. Note that in the language of (\cite{K},subsection 4.1) this means that $\varphi^{(0,0)}$ defines a map $\varphi: \tilde{\Q}^-\to\tilde{\Q}^+$, where $\tilde{\Q}^{\pm}$ are viewed as formal pointed manifolds rather than linear spaces. In particular, it follows that $\varphi^{(0,0)}$ descends to a map between the spaces of functions,  $$\varphi^{(0,0)}: \T^{(0,0)}\tilde{\Q}^+_{\tilde{X}^+}\to \T^{(0,0)}\tilde{\Q}^-_{\tilde{X}^-}.$$ On the other hand, the map $\varphi^{(0,0)}$ uniquely defines a corresponding map $\varphi^{(0,1)}:\T^{(0,1)}\tilde{\Q}^+\to\T^{(0,1)}\tilde{\Q}^-$ on the space of one-forms by the requirement that $d\circ\varphi^{(0,0)}=\varphi^{(0,1)}\circ d$ with the exterior derivative $d:\T^{(0,0)}\tilde{\Q}^{\pm}\to\T^{(0,1)}\tilde{\Q}^{\pm}$. It is given by $$\varphi^{(0,1)}(q^{\Gamma_+} dq_{\gamma^+})=\varphi^{(0,0)}(q^{\Gamma_+})\varphi^{(0,1)}(dq_{\gamma^+})$$ with $\varphi^{(0,1)}(dq_{\gamma^+})=d(\varphi^{(0,0)}(q_{\gamma^+}))$. Using the compatibility of exterior derivative and Lie derivative, it again follows that $\varphi^{(0,1)}$ descends to a map between the spaces of one-forms,  $$\varphi^{(0,1)}: \T^{(0,1)}\tilde{\Q}^+_{\tilde{X}^+}\to \T^{(0,1)}\tilde{\Q}^-_{\tilde{X}^-}.$$ 

The reader familiar with the algebraic framework of SFT observes that the chain maps for full contact homology provides us with examples for morphisms of differential graded manifolds. Indeed, the invariance properties of contact homology stated in the theorem above, lead to an isomorphism of the corresponding differential graded manifolds. Note that in the last section we used the same proof to show that the $L_{\infty}$-structure on the cylindrical contact homology $\HC_*^{\cyl}(M_{\phi})$ is well-defined up to homotopy. For this we used that the $L_{\infty}$-structures that has orginally been defined on the chain space, actually descends to an $L_{\infty}$-structure on homology by homotopy transfer. Using that the cohomological vector field $\tilde{X}$ is just a geometric way to encode the $L_{\infty}$-structure, see \cite{K}, in complete analogy we obtain that there exists a cohomological vector field $X$ on the sum of the Floer cohomologies $\Q:=\HC_*^{\cyl}(M_{\phi})=\bigoplus_k\HF_*(\phi^k)$ such that $(\tilde{\Q},\tilde{X})$ and $(\Q,X)$ define the same differential graded manifold, up to homotopy; in particular, we have that the tensor fields are the same, $$\T^{(r,s)}\Q_{X}\cong \T^{(r,s)}\tilde{\Q}_{\tilde{X}}.$$  For this observe that the contact homology differential naturally can be written as an infinite sum, $\tilde{X}=\sum_{r=1}^{\infty} \tilde{X}_r$, where $\tilde{X}_r\in\T^{(1,0)}\tilde{\Q}$ contains only those summands with $q_{\gamma}$-monomials of length $r$. It is an important observation that, in our case of Hamiltonian mapping tori, this sum indeed starts with $r=1$, since there obviously are no $\tilde{J}$-holomorphic disks in $\RS\times M$. On the other hand, the first summand $\tilde{X}_1$ agrees with the differential $\del$ in cylindrical contact homology, $$\tilde{X}_1=\sum_{\gamma^+}\Bigl(\sum_{\gamma^-,A} \frac{1}{\kappa_{\gamma^-}}\sharp\IM^{\gamma^+}_{\gamma^-}(A)\cdot\; q_{\gamma^-} t^{c_1(A)}\Bigr)\frac{\del}{\del q_{\gamma^+}}\;\in\; \T^{(1,0)}\tilde{\Q}.$$ 

\noindent\emph{Evaluation maps for tensor fields}\\

We would like to end with the following statement about differential graded manifolds. In order to justify that we can speak about tensor fields on some generalized version of manifold, we would like to show there is indeed a evaluation map for tensor fields in $\T^{(r,s)}\tilde{\Q}_{\tilde{X}}$ . 

\begin{proposition} Since $\tilde{X}(0)=0$, there exists a natural evaluation (or restriction) map for tensor fields at $q=0$, $$\T^{(r,s)}\tilde{\Q}_{\tilde{X}}\;\to\;((\HC_*^{\cyl})^{\otimes s})^*\otimes (\HC_*^{\cyl})^{\otimes r},\, \alpha\mapsto \alpha(0).$$ \end{proposition}

\begin{proof}
For the proof we show that the Lie derivative of a tensor field $\alpha\in\T^{(r,s)}\tilde{\Q}$ in the direction of $\tilde{X}$ at $q=0$ can be computed from the restriction of the tensor $\alpha(0)$ at $q=0$  and the cylindrical contact homology differential $\del: T_0\tilde{\Q}\to T_0\tilde{\Q}$ by $$(\LL_{\tilde{X}}\alpha)(0) \;=\; \alpha(0)\circ\del \,-\, \del\circ\alpha(0),$$ where $\del$ denotes the obvious extension (using Leibniz rule) of the boundary operator to the tensor products $T_0\tilde{\Q}^{\otimes r}$, $T_0\tilde{\Q}^{\otimes s}$. \\

First, using the definition of the Lie derivative for tensor fields, we find that
\begin{eqnarray*}
&&(\LL_{\tilde{X}}\alpha)\Big(\frac{\del}{\del q_{\gamma_0}}\otimes\ldots\otimes dq_{\gamma^+_s} \Big) 
\end{eqnarray*}
is given by 
\begin{eqnarray*}
&&\Big(\tilde{X}\Big(\alpha\Big(\frac{\del}{\del q_{\gamma_0}}\otimes\ldots\otimes dq_{\gamma^+_s} \Big)\Big)\Big)
\\&&-\; \Big(\alpha\Bigl(\LL_{\tilde{X}}\frac{\del}{\del q_{\gamma_0}}\otimes\ldots\otimes dq_{\gamma^+_s}\Big)\Big)
\\&&-\; (\ldots) 
\\&&-\; (-1)^{|q_{\gamma_0}|+\ldots-|q_{\gamma^+_{s-1}}|} \Big(\alpha\Bigl(\frac{\del}{\del q_{\gamma_0}}\otimes\ldots\otimes\LL_{\tilde{X}}dq_{\gamma^+_s}\Bigr)\Big).
\end{eqnarray*}
Now employing that $\tilde{X}(0)=0$, we find that the first summand involving the derivative of $\alpha$ vanishes and only the other summands involving only the value of the tensor field at the point zero, 
\begin{eqnarray*} 
&&-\; \alpha(0)\Big(\Big(\LL_{\tilde{X}}\frac{\del}{\del q_{\gamma_0}}\Big)(0)\otimes\ldots\otimes dq_{\gamma^+_s}\Big)
\\&&-\; (\ldots)
\\&&-\; (-1)^{|q_{\gamma_0}|+\ldots-|q_{\gamma_{s-1}}|} \alpha(0)\Bigl(\frac{\del}{\del q_{\gamma_0}}\otimes\ldots\otimes (\LL_{\tilde{X}} dq_{\gamma^+_s})(0)\Bigr),
\end{eqnarray*}
remain. With the observation that the cylindrical contact homology differential $\del$ is given by 
\begin{eqnarray*} 
\frac{\del}{\del q_{\gamma^-}}\;\mapsto\; \Big(\LL_{\tilde{X}}\frac{\del}{\del q_{\gamma^-}}\Big)(0)&=&\frac{\del \tilde{X}}{\del q_{\gamma^-}}(0)\\&=&\,\sum_{\gamma^+,A}\frac{1}{\kappa_{\gamma^-}}\sharp\IM^{\gamma^+}_{\gamma^-}(A)/\IR\; t^{c_1(A)}\cdot\,\frac{\del}{\del q_{\gamma^+}}, 
\end{eqnarray*} 
the claim follows.
\end{proof}

The proof suggests that an analogous result can be shown when the point $q=0$ is replaced by any other point $q\in\tilde{\Q}$ where the cohomological vector field vanishes, $\tilde{X}(q)=0$. Using the relation between the cohomological vector field and the $L_{\infty}$-structure discussed above, note that the vanishing of $\tilde{X}$ is equivalent to requiring that $q\in\tilde{\Q}$ is a solution of the Maurer-Cartan equation for the $L_{\infty}$-structure. The boundary operator $\del: T_0\tilde{\Q}\to T_0\tilde{\Q}$ of cylindrical contact homology accordingly needs to be replaced by a deformed version of it, 
$$\del_q:\; T_q\tilde{\Q}\to T_q\tilde{\Q},\;\frac{\del}{\del q_{\gamma}}\mapsto\frac{\del \tilde{X}}{\del q_{\gamma}}(q)\;=\;d\tilde{X}(q)\Big(\frac{\del}{\del q_{\gamma}}\Big).$$ There is however a small caveat: Note that, for $\del_q$ to be of pure degree (one), we have to require that the degree of $q$ as well as the degree of the formal variable $t$ are even. \\

In order to elaborate this picture further, let us assume for a moment that the symplectic manifold $(M,\omega)$ is Calabi-Yau in the sense that $c_1(A)=0$ for all $A\in H_2(M)$. In this case it follows that the Conley-Zehnder index is independent of the chosen spanning surfaces, so that we can lift the $\IZ_2$-grading on $\tilde{\Q}=\bigoplus_k\CF_*(\phi^k)$ to a $\IZ$-grading. Using this integer grading we can then further define a so-called \emph{Euler vector field} $E\in\T^{(1,0)}\tilde{\Q}$ given by $$\tilde{E}\;=\;\sum_{\gamma} |q_{\gamma}|\cdot\; q_{\gamma}\frac{\del}{\del q_{\gamma}}\;\;\textrm{with}\;\;|q_{\gamma}|=-\CZ(\gamma)-2(\dim M-2).$$ The fact that the cohomological vector field $\tilde{X}$ is counting holomorphic curves with Fredholm index one translates into the algebraic fact that $\LL_{\tilde{E}} \tilde{X}=[\tilde{E},\tilde{X}]=\tilde{X}$. 

\begin{corollary} For every point $q\in\tilde{\Q}$ in the chain space of cylindrical contact homology where the cohomological vector field $\tilde{X}\in\T^{(1,0)}\tilde{\Q}$ as well as the Euler vector field $\tilde{E}\in\T^{(1,0)}\tilde{\Q}$ vanish, $\tilde{X}(q)=0=\tilde{E}(q)$, we can define a deformed version of cylindrical contact homology $\HC_*^{\cyl,q}(M_{\phi})=H_*(T_q\tilde{\Q},\del_q)$ with the same chain space $T_q\tilde{\Q}\cong T_0\tilde{\Q}\cong\tilde{\Q}$ but deformed differential given by $\del_q=d\tilde{X}(q)$, $\del/\del q_{\gamma}\mapsto (d\tilde{X}/\del q_{\gamma})(q)$. Furthermore there again exists a natural evaluation map for tensor fields at $q$, $$\T^{(r,s)}\tilde{\Q}_{\tilde{X}}\;\to\;((\HC_*^{\cyl,q})^{\otimes s})^*\otimes (\HC_*^{\cyl,q})^{\otimes r},\, \alpha\mapsto \alpha(q).$$ \end{corollary}
  
\begin{proof} For the proof it remains to show that the deformed differential $\del_q$ gives indeed a boundary operator on the chain space $T_q\tilde{\Q}\cong\tilde{\Q}$ of cylindrical contact homology. For this we apply the computation from above to the tensor field $d\tilde{X}=\sum_{\gamma} (\del\tilde{X}/\del q_{\gamma})\cdot dq_{\gamma}\in\T^{(1,1)}\tilde{\Q}$. First, since $\LL_{\tilde{X}}\tilde{X}=[\tilde{X},\tilde{X}]=0$ and the exterior derivative $d$ commutes with the Lie derivative $\LL_{\tilde{X}}$, we indeed obtain that $\LL_{\tilde{X}}(d\tilde{X})=0$. Since $\tilde{X}(q)=0$, it follows that $$d\tilde{X}(q)\Big(\Big(\LL_{\tilde{X}}\frac{\del}{\del q_{\gamma}}\Big)(q)\otimes dq_{\gamma^+}\Big)\,+\, d\tilde{X}(q)\Big(\frac{\del}{\del q_{\gamma}}\otimes (\LL_{\tilde{X}}dq_{\gamma^+})(q)\Big)\;=\;0$$ for all closed orbits $\gamma$, $\gamma^+$. But since the latter expression can be shown to agree with $2 \cdot (d\tilde{X}(q)\circ d\tilde{X}(q))(\del/\del q_{\gamma}\otimes dq_{\gamma^+})$, the claim follows. \end{proof}

\section{Big pair-of-pants product}

As already mentioned, it is the main goal of this paper to generalize the big quantum product on quantum homology to a big pair-of-pants product in Hamiltonian Floer theory. As with the $L_{\infty}$-structure and the differential graded manifold structure, we define the corresponding big pair-of-pants product and the resulting cohomology F-manifold structure on the cylindrical contact homology $\HC^{\cyl}_*(M_{\phi})$ which agrees with the sum of the Floer cohomologies $\HF_*(\phi^k)$, $k\in\IN$ . Furthermore note that, as the new moduli spaces are going to be defined as subsets of the moduli spaces from full contact homology used in the definition of $\tilde{X}$, \emph{we immediately assume that $M$ is a symplectic manifold with contact-type boundary and the Hamiltonian function has asymptotic linear slope in the cylindrical end of the completion.} \\

We start by recalling, see proposition \ref{Lie} and also (\cite{F1}, proposition 2.2), that the map $\tilde{u}=(h,u):\Si\to\IR\times M_{\phi}\cong\RS\times M$ is $\tilde{J}$-holomorphic precisely when $h:\Si\to\RS$ is holomorphic and $u:\Si\to M$ satisfies the Floer equation $\CR_{J,H,h}(u)=\Lambda^{0,1}(du+X^H_{h_2}\otimes dh_2)$. Fixing the complex structure $j$ on $\Si$, that is, the positions of the punctures $z_0,\ldots,z_{r-1}\in\IC$, and forgetting the map $u$, we are hence left with holomorphic functions $h$ on $\IC$ with $r$ zeroes $z_0,\ldots,z_{r-1}$ of predescribed orders $k_0,\ldots,k_{r-1}$. Since such a holomorphic map always exists and is uniquely determined up to a $\IC^*\cong\RS$-factor, it follows that, after dividing out the natural $\IR$-action in the target, there is an $S^1$-family of maps $h$, which becomes visible in terms of the asymptotic markers. For the latter observe that the asymptotic markers are fixed by the special points on the closed orbits. Recall that we choose on each $k$-periodic orbit $k$ special points (each with weight $1/k$ to cure for the resulting overcounting) which are naturally given by the intersection of the orbit with the fibre over $\{0\}\times M\subset S^1\times M\cong M_{\phi}$. \\
 
In contrast to the moduli spaces considered in \cite{R}, we hence see that the conformal structure on the punctured Riemann sphere is allowed to vary and we allow the asymptotic markers above and below to rotate simultaneously. Generalizing this, the new moduli spaces used to define the \emph{big pair-of-pants product} shall consist of Floer solutions $u:\Si\to M$ with an arbitrary number of negative punctures and varying conformal structure but with fixed asymptotic markers at all punctures (depending on the underlying conformal structure). Reduced to the essence, for the definition of the new moduli spaces for the big pair-of-pants product we going to use that the fact that, in contrast the well-known case of contact manifolds, there exists a natural projection from $M_{\phi}\cong S^1\times M$ to the circle. Instead of fixing the asymptotic markers, we equivalently fix the induced branched covering map $h=(h_1,h_2):\Si\to\RS$. 

\begin{center}
\includegraphics[height=4cm]{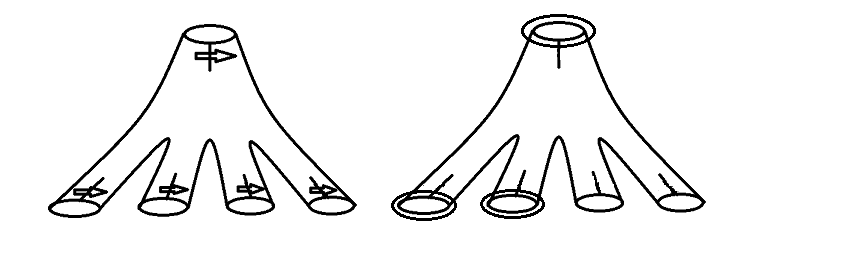}\\
\small{Holomorphic curves for the cohomological vector field (left) and for the big pair-of-pants product (right)}
\end{center}

For this we introduce an additional marked point $z^*$ on the underlying punctured Riemann sphere $\Si=S^2\backslash\{z_0,\ldots,z_{r-1},\infty\}$. As in \cite{EGH} we denote by $\IM^{\gamma^+}_1(\Gamma)$ the moduli space of $\tilde{J}$-holomorphic curves $(\tilde{u},z_0,\ldots,z_{r-1},z^*)$ in $\RS\times M$ with one (unconstrained) additional marked point. Using the natural diffeomorphism $M_{\phi}\cong S^1\times M$, note that there exists an evaluation map from $\IM^{\gamma^+}_1(\Gamma)$ to $S^1$, given by mapping $((h,u),z_0,\ldots,z_{r-1},z^*)$ to $h_2(z^*).$ In order to fix the asymptotic markers and hence the holomorphic map $h$ using the latter evaluation map, it then remains to constrain the additional marked point a priori without using the map. \\

We first describe how the well-known pair-of-pants product on Floer homology appears in our framework. For this observe that the pair-of-pants product gives unique maps $\star_0: \HF_*(\phi^k)\otimes\HF_*(\phi^{\ell})\to\HF_*(\phi^{k+\ell})$ for all $k,\ell\in\IN$, see e.g. \cite{Sch} and \cite{MDSa}. Since $\HC_*^{\cyl}(M_{\phi})\cong\bigoplus_k\HF_*(\phi^k)$, it immediately follows that the pair-of-pants product defines a product on the cylindrical contact homology of every Hamiltonian mapping torus. \\

Observe that the moduli space $\IM^{\gamma^+}(\gamma_0,\gamma_1)$ can equivalently be defined as the set of $\tilde{J}$-holomorphic maps $\tilde{u}:\Si\to M$ starting from the three punctured sphere $\Si=S^2\backslash\{0,1,\infty\}$. Using these unique coordinates, we can fix the position of the additional marked $z^*$ on $\Si$ a priori for each moduli space $\IM^{\gamma^+}(\gamma_0,\gamma_1)$. With this we define the submoduli space $\IM^{\gamma^+}_{\gamma_0,\gamma_1}\subset\IM^{\gamma^+}_1(\gamma_0,\gamma_1)$ of $\tilde{J}$-holomorphic maps $((h,u),z^*)\in\IM^{\gamma^+}_1(\gamma_0,\gamma_1)$ where the additional marked $z^*$ is constrained using the unique coordinates given by the three punctures $z_0=0$, $z_1=1$ (and $z_{\infty}=\infty$) and required to get mapped to $0\in S^1$ under the map $h_2$.  

\begin{proposition} Using the new moduli spaces $\IM^{\gamma^+}_{\gamma_0,\gamma_1}$ and the natural identification of the chain space of cylindrical contact homology with the sum of the chain spaces of the Floer cohomologies, the pair-of-pants product can be defined by $$\gamma_0\;\tilde{\star}_0\; \gamma_1 \;=\; \sum_{\gamma^+,A} \frac{1}{\kappa_{\gamma_0}\kappa_{\gamma_1}}\cdot\frac{1}{k}\cdot \#\IM_{\gamma_0,\gamma_1}^{\gamma^+}(A) \cdot\; \gamma^+ t^{c_1(A)}.$$  \end{proposition}

\begin{proof} Following the above discussion, see also proposition \ref{Lie}, the moduli space $\IM^{\gamma^+}(\Gamma)$ can be identified with the moduli space of Floer solutions $u:\Si\to M$ starting from the three-punctured sphere with unconstrained asymptotic marker at the positive puncture which in turn constrains the asymptotic markers at the two negative punctures via the induced map $h$. While the submoduli space $\IM^{\gamma^+}_{\gamma_0,\gamma_1}\subset\IM^{\gamma^+}(\gamma_0,\gamma_1)$ is precisely characterized by the fact that the map $h$ is fixed, note that there is just a $k$-to-one correspondence between $\tilde{J}$-holomorphic curves in $\IM^{\gamma^+}_{\gamma_0,\gamma_1}$ and Floer solutions $u:\Si\to M$ with fixed asymptotic markers at all punctures, since there is just a $k$-to-one correspondence between asymptotic markers at the positive puncture and corresponding maps $h$ to the cylinder. 
Note that in our symmetrized definition, we have $k$ special points on the $k$-periodic orbit $\gamma^+$ given by the intersection with the fibre $\{0\}\times M \subset  S^1\times M\cong M_{\phi}$ which in turn define $k$ asymptotic markers, all with weight $1/k$. This said, the coefficient $1/k$ simply means that we count elements in the quotient $\IM^{\gamma^+}_{\gamma_0,\gamma_1}/\IZ_k$ under the natural $\IZ_k$-action given by rotating the asymptotic marker at the positive puncture by $1/k\in S^1=\IR/\IZ$. \end{proof}

When there are more than three marked points on the sphere, such canonical coordinates only exist whenever one selects three marked points from the given $r$ marked points. \emph{It is the key observation for our definition of the big pair-of-pants product that (as for the big quantum product but unlike for the Gromov-Witten potential) such choice of three special punctures (the positive puncture and two of the negative punctures) is natural.} \\

To this end, without loss of generality, let us assume that we use the first two and the last marked point to define coordinates by setting $z_0=0$, $z_1=1$ and $z_{\infty}=\infty$. 

\begin{definition} After fixing the position of the additional marked point $z^*$ using the unique coordinates on punctured sphere given by the three special punctures, the new moduli spaces $\IM^{\gamma^+}_{\gamma_0,\gamma_1}(\Gamma)\subset\IM^{\gamma^+}(\gamma_0,\gamma_1,\Gamma)$ consists of those (equivalence classes of) tuples $(\tilde{u},0,1,z_2,\ldots,z_{r-1})$ with $\tilde{u}=(h_1,h_2,u):\Si\to\RS\times M$, where $z^*$ gets mapped to $0\in S^1$ under $h_2:\Si\to S^1$. \end{definition} 

In order to obtain a product that is already (graded) commutative on the chain level, let us from now on assume that $z^*$ is chosen symmetrically in such a way that it does not change when the roles of $z_0$ and $z_1$ (in fixing the coordinates) are interchanged. While in the case of three punctures from before we could directly use the resulting coordinates to fix the position of the additional marked point for every moduli space $\IM^{\gamma^+}_1(\gamma_0,\gamma_1)$, in the case of more than three punctures we need to make a genericity assumption. The reason is that, in contrast to the case of only three punctures, now we have to explicitly exclude that the additional marked point coincides with one of the remaining punctures as the $\tilde{J}$-holomorphic curve varies inside the moduli space. In order to see that this is possible, it is crucial to observe that we are only interested moduli spaces of real dimension zero or one. \\

\begin{center}
\includegraphics[height=4cm]{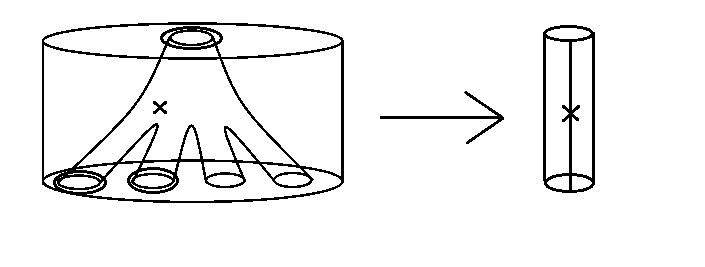}\\
\small{Using three special punctures we can fix a special point $z^*$ and require that it gets mapped to $0\in S^1$.}
\end{center}

In the appendix we show that, by allowing the Hamiltonian $H$ and the compatible almost complex structure $J$ to vary in a domain-dependent way, the forgetful map from the resulting universal moduli space $\widetilde{\IM}^{\gamma^+}_{\gamma_0,\gamma_1}(\Gamma)$ to the moduli space $\IM_{r+1}$ of complex structures on the punctured Riemann surface is a submersion. In particular, it follows that the image of  $\widetilde{\IM}^{\gamma^+}_{\gamma_0,\gamma_1}(\Gamma)$ transversally meets the stratum $\Delta$ on $\IM_{r+1}$, defined as the set of complex structures where the additonal marked point $z^*$ agrees with one of the punctures. Since, by Sard's theorem, this continues to hold for generic choices of domain-dependent Hamiltonians and compatible almost complex structures, we find that the at most one-dimensional image of the moduli space $\IM^{\gamma^+}_{\gamma_0,\gamma_1}(\Gamma)$ under the forgetful map does not meet the forbidden stratum $\Delta\subset\IM_{r+1}$, possibly after perturbing $H$ and $J$. \\

Note that the new moduli spaces $\IM^{\gamma^+}_{\gamma_0,\gamma_1}(\Gamma)$ are indeed  generalizations of the moduli spaces $\IM^{\gamma^+}_{\gamma_0,\gamma_1}$ in the sense that  $\IM^{\gamma^+}_{\gamma_0,\gamma_1}=\IM^{\gamma^+}_{\gamma_0,\gamma_1}(\emptyset)$. Summarizing, this motivates the following 

\begin{definition} On the chain level, the \emph{big pair-of-pants product} is defined to be the $(1,2)$-tensor field $\tilde{\star}\in\T^{(1,2)}\tilde{\Q}$ given by $$\sum_{\gamma^+,\gamma_0,\gamma_1}\Bigl(\sum_{\Gamma,A} \frac{1}{(r-2)!} \frac{1}{\kappa_{\gamma_0}\kappa_{\gamma_1}\kappa^{\Gamma}} \cdot\frac{1}{k}\cdot\#\IM^{\gamma^+}_{\gamma_0,\gamma_1}(\Gamma,A)\cdot \;t^{c_1(A)}q^{\Gamma}\Bigr)dq_{\gamma_0}\otimes dq_{\gamma_1}\otimes\frac{\del}{\del q_{\gamma^+}}$$ where $k$ is the period of the closed orbit $\gamma^+$. \end{definition}

In analogy to Gromov-Witten theory, from $\IM^{\gamma^+}_{\gamma_0,\gamma_1} \;=\; \IM^{\gamma^+}_{\gamma_0,\gamma_1}(\emptyset)$ it follows that the big pair-of-pants product is indeed a deformation of the classical pair-of-pants product in the sense that at $q=(q_{\gamma})=0$ it agrees with the small product, $$\frac{\del}{\del q_{\gamma_0}}\;\tilde{\star}_0\;\frac{\del}{\del q_{\gamma_1}}\;=\; \sum_{\gamma^+}\Bigl(\sum_A\frac{1}{\kappa_{\gamma_0}\kappa_{\gamma_1}}\cdot\frac{1}{k}\cdot\#\IM^{\gamma^+}_{\gamma_0,\gamma_1}(A)\cdot\;t^{c_1(A)}\Bigr)\frac{\del}{\del q_{\gamma^+}}.$$ For this recall that we again identify the chain space of cylindrical contact homology with the tangent space $T_0\tilde{\Q}$ at zero by identifying each closed Reeb orbit $\gamma$ with the tangent vector $\del/\del q_{\gamma}$.  \\

\section{Cohomology F-manifolds}

\noindent\emph{Master equation}\\

Like for the small pair-of-pants product, we can only expect the big pair-of-pants product to satisfy algebraic properties like associativity and commutativity when viewing it as an element in some homology. The main step is to show that the big pair-of-pants product indeed defines an element in some homology: Indeed, it descends to a vector field product in $\T^{(1,2)}\tilde{\Q}_{\tilde{X}}$ on the differential graded manifold $\tilde{\Q}_{\tilde{X}}=(\tilde{\Q},\tilde{X})$.

\begin{proposition}\label{master} The big pair-of-pants product $\tilde{\star} \in \T^{(1,2)}\tilde{\Q}$ and the cohomological vector field $\tilde{X}\in\T^{(1,0)}\tilde{\Q}$ from contact homology satisfy $$\LL_{\tilde{X}} \tilde{\star} = 0.$$ \end{proposition}

Of course, the proof of the theorem relies on the translation from geometry into algebra of a compactness result for the new moduli spaces. This is the content of the following lemma. We emphasize that all occuring products of moduli spaces are to be understood as direct products as in \cite{EGH}, see also the appearance of combinatorical factors in the subsequent proof of the theorem.    

\begin{lemma} By counting broken $\tilde{J}$-holomorphic curves (with signs) in the codimension-one boundary of the moduli space $\IM^{\gamma^+}_{\gamma_0,\gamma_1}(\Gamma)$ one obtains that the sum of the following terms is equal to zero,
\begin{enumerate}
\item $$\sum_{\gamma\in\Gamma'} \frac{1}{\kappa_{\gamma_0}\kappa_{\gamma_1}\kappa^{\Gamma'}}\frac{1}{k}\#\IM^{\gamma^+}_{\gamma_0,\gamma_1}(\Gamma')\cdot\frac{1}{\kappa^{\Gamma''}}\#\IM^{\gamma}(\Gamma'')$$
\item $$\sum_{\gamma_0\in\Gamma''} \frac{1}{\kappa_{\gamma'}\kappa_{\gamma_1}\kappa^{\Gamma'}}\frac{1}{k}\#\IM^{\gamma^+}_{\gamma',\gamma_1}(\Gamma')\cdot\frac{1}{\kappa_{\Gamma''}}\#\IM^{\gamma}(\Gamma'')$$ 
\item $$\sum_{\gamma_1\in\Gamma''}\frac{1}{\kappa_{\gamma_0}\kappa_{\gamma'}\kappa^{\Gamma'}}\frac{1}{k}\#\IM^{\gamma^+}_{\gamma_0,\gamma'}(\Gamma')\cdot\frac{1}{\kappa_{\Gamma''}}\#\IM^{\gamma'}(\Gamma'')$$ 
\item $$\sum_{\gamma\in\Gamma'}\frac{1}{\kappa^{\Gamma'}}\#\IM^{\gamma^+}(\Gamma')\cdot\frac{1}{\kappa_{\gamma_0}\kappa_{\gamma_1}\kappa^{\Gamma''}}\frac{1}{k'}\#\IM^{\gamma} _{\gamma_0,\gamma_1}(\Gamma''),$$ 
\end{enumerate}
where  $k$ ($k'$) is the period of $\gamma^+$ ($\gamma'$) and we take the union over all $\Gamma'$, $\Gamma''$ whose union (apart from the special orbits explicitly mentioned) is $\Gamma$.  
\end{lemma}

\begin{proof} First, it is just a combinatorical exercise to deduce from the compactness result for the moduli space for contact homology stated above that the codimension one boundary of the moduli space $\IM^{\gamma^+}(\gamma_0,\gamma_1,\Gamma)$ has the corresponding components 
\begin{enumerate} 
\item $\IM^{\gamma^+}(\gamma_0,\gamma_1,\Gamma')\times\IM^{\gamma}(\Gamma'')$ with $\gamma\in\Gamma'$, 
\item $\IM^{\gamma^+}(\gamma,\gamma_1,\Gamma')\times\IM^{\gamma}(\Gamma'')$ with $\gamma_0\in\Gamma''$,
\item $\IM^{\gamma^+}(\gamma_0,\gamma,\Gamma')\times\IM^{\gamma}(\Gamma'')$ with $\gamma_1\in\Gamma''$,
\item $\IM^{\gamma^+}(\Gamma')\times\IM^{\gamma}(\gamma_0,\gamma_1,\Gamma'')$ with $\gamma\in\Gamma'$. 
\end{enumerate} 
For this observe that after splitting up into a two-level holomorphic curve either the three special punctures still lie on the same component (which leads to components of type 1) or there are two special punctures on one component and one special puncture on the other component (which leads to components of type 2, 3 and 4). After introducing the constrained additional marked point, it follows that it sits on the unique component which carries two or three of the special punctures. Denoting by $\IM^{\gamma^+}_{\gamma_0,\gamma_1}(\Gamma')\subset \IM^{\gamma^+}_1(\gamma_0,\gamma_1,\Gamma')$ (and so on) the moduli space of holomorphic curves with additional marked point constrained by the three special punctures and mapping to zero as before, it follows that the codimension-one boundary of the zero set $\IM^{\gamma^+}_{\gamma_0,\gamma_1}(\Gamma)$ has the analogous components
\begin{enumerate} 
\item $\IM^{\gamma^+}_{\gamma_0,\gamma_1}(\Gamma')\times\IM^{\gamma}(\Gamma'')$ with $\gamma\in\Gamma'$, 
\item $\IM^{\gamma^+}_{\gamma,\gamma_1}(\Gamma')\times\IM^{\gamma}(\Gamma'')$ with $\gamma_0\in\Gamma''$,
\item $\IM^{\gamma^+}_{\gamma_0,\gamma}(\Gamma')\times\IM^{\gamma}(\Gamma'')$ with $\gamma_1\in\Gamma''$,
\item $\IM^{\gamma^+}(\Gamma')\times\IM^{\gamma}_{\gamma_0,\gamma_1}(\Gamma'')$ with $\gamma\in\Gamma'$. 
\end{enumerate}
Concerning the combinatorical factors in the statement, observe that the multiplicities $k$ ($k'$) of $\gamma^+$ ($\gamma'$) show up because induced holomorphic maps to the cylinders are $k$- ($k'$-) fold coverings. More precisely, after defining $\IM^{\gamma^+}_{\gamma_0,\gamma_1}(\Gamma)\subset \IM^{\gamma^+}(\gamma_0,\gamma_1,\Gamma)$ by mapping the constrained additional marked point to zero, we still have $k$ possible choices for the asymptotic markers at the positive punctures. We emphasize that the underlying moduli spaces indeed naturally carry a $\IZ_k$-symmetry (due to our symmetrized definition). \end{proof}

For the proof of the theorem it remains to translate the geometrical result of the lemma into algebra. 

\begin{proof}\emph{(of the theorem)} Using the definition of the Lie derivative of higher tensors, we obtain for any choice of basis vectors $\del/\del q_{\gamma_0}$, $\del/\del q_{\gamma_1} \in \T^{(1,0)}\tilde{\Q}$ and forms $dq_{\gamma^+}\in\T^{(0,1)}\tilde{\Q}$ that 
\begin{eqnarray*} 
(\LL_{\tilde{X}}\tilde{\star})\Bigl(\frac{\del}{\del q_{\gamma_0}}\otimes \frac{\del}{\del q_{\gamma_1}}\otimes dq_{\gamma^+}\Bigr)
&=& \LL_{\tilde{X}}\Bigl(\tilde{\star}\Bigl(\frac{\del}{\del q_{\gamma_0}}\otimes \frac{\del}{\del q_{\gamma_1}}\otimes dq_{\gamma^+}\Bigr)\Bigr)\\
&-& \tilde{\star}\Bigl(\LL_{\tilde{X}}\frac{\del}{\del q_{\gamma_0}}\otimes \frac{\del}{\del q_{\gamma_1}}\otimes dq_{\gamma^+}\Bigr)\\
&-& (-1)^{|q_{\gamma_0}|} \tilde{\star}\Bigl(\frac{\del}{\del q_{\gamma_0}}\otimes \LL_{\tilde{X}}\frac{\del}{\del q_{\gamma_1}}\otimes dq_{\gamma^+}\Bigr)\\
&-& (-1)^{|q_{\gamma_0}|+|q_{\gamma_1}|} \tilde{\star}\Bigl(\frac{\del}{\del q_{\gamma_0}}\otimes \frac{\del}{\del q_{\gamma_1}}\otimes \LL_{\tilde{X}}dq_{\gamma^+}\Bigr).
\end{eqnarray*}
Now using the definition of $\tilde{\star}\in\T^{(1,2)}\tilde{\Q}$, 
$$\tilde{\star}\Bigl(\frac{\del}{\del q_{\gamma_0}}\otimes \frac{\del}{\del q_{\gamma_1}}\otimes dq_{\gamma^+}\Bigr) \;=\; 
   \sum_{\Gamma,A}\frac{1}{(r-2)!} \frac{1}{\kappa_{\gamma_0}\kappa_{\gamma_1}\kappa^{\Gamma}}\frac{1}{k}\#\IM^{\gamma^+}_{\gamma_0,\gamma_1}(\Gamma,A)\cdot\; q^{\Gamma}t^{c_1(A)}$$
we find that the first summand is given by 
\begin{eqnarray*}
&&\LL_{\tilde{X}}\bigl(\sum_{\Gamma,A}\frac{1}{(r-2)!}\frac{1}{\kappa_{\gamma_0}\kappa_{\gamma_1}\kappa^{\Gamma}}\frac{1}{k} \#\IM^{\gamma^+}_{\gamma_0,\gamma_1}(\Gamma,A)\cdot\; q^{\Gamma}t^{c_1(A)}\bigr)\\
&&= \sum_{\Gamma,A}\frac{1}{(r-2)!}\frac{1}{\kappa_{\gamma_0}\kappa_{\gamma_1}\kappa^{\Gamma}}\frac{1}{k} \#\IM^{\gamma^+}_{\gamma_0,\gamma_1}(\Gamma,A)\cdot\;\LL_{\tilde{X}} q^{\Gamma}\cdot t^{c_1(A)}. \end{eqnarray*}
Together with $$\LL_{\tilde{X}} q_{\gamma}= \tilde{X}(q_{\gamma}) = \sum_{\Gamma,A}\frac{1}{r!}\frac{1}{\kappa^{\Gamma}}\sharp\IM^{\gamma}(\Gamma,A)\cdot\; q^{\Gamma} t^{c_1(A)}$$ and using the Leibniz rule, we find that the first summand is precisely counting the boundary components of type 1,  $\IM^{\gamma^+}_{\gamma_0,\gamma_1}(\Gamma')\times\IM^{\gamma}(\Gamma'')$. For the combinatorical factors, observe that there are $\kappa_{\gamma}$ ways to glue two holomorphic curves along a multiply-covered orbit $\gamma$. \\

For the other summands, we can show in the same way that they correspond to the other boundary components. \\

Indeed, using $$\LL_{\tilde{X}}\frac{\del}{\del q_{\gamma_0}}\;=\;\frac{\del \tilde{X}}{\del q_{\gamma_0}}\;=\;\sum_{\gamma}\sum_{\Gamma,A}\frac{1}{r!}\frac{1}{\kappa_{\gamma_0}\kappa^{\Gamma}}\sharp\IM^{\gamma}((\gamma_0,\Gamma),A)\;\cdot\;q^{\Gamma} t^{c_1(A)}\cdot\frac{\del}{\del q_{\gamma}}$$ (and similar for $\gamma_1$), it follows that the second and the third summand correspond to boundary components $\IM^{\gamma^+}_{\gamma,\gamma_1}(\Gamma')\times\IM^{\gamma}(\Gamma'')$ and $\IM^{\gamma^+}_{\gamma_0,\gamma}(\Gamma')\times\IM^{\gamma}(\Gamma'')$ with $\gamma_0,\gamma_1\in\Gamma''$ of type 2. For the combinatorical factors we refer to the remark above. \\

Finally, using 
\begin{eqnarray*}
&&(\LL_{\tilde{X}}dq_{\gamma^+})\Bigl(\frac{\del}{\del q_{\gamma}}\Bigr)\;=\; -(-1)^{|q_{\gamma^+}|} \; dq_{\gamma^+}\Bigl(\frac{\del \tilde{X}}{\del q_{\gamma}}\Bigr)\\ 
&&=\; -(-1)^{|q_{\gamma^+}|} \;\sum_{\Gamma,A} \frac{1}{r!}\frac{1}{\kappa_{\gamma}\kappa^{\Gamma}}\sharp\IM^{\gamma^+}((\gamma,\Gamma),A)\;\cdot\; q^{\Gamma} t^{c_1(A)}
\end{eqnarray*} 
we find that the last summand corresponds to boundary components $\IM^{\gamma^+}(\Gamma')\times\IM^{\gamma}_{\gamma_0,\gamma_1}(\Gamma'')$ with $\gamma\in\Gamma'$ of type 3, where the combinatorical factors are treated as above.
\end{proof}

We now give the main definition of this paper. Recall that for every differential graded manifold $(\tilde{\Q},\tilde{X})$ there exists a well-defined space $\T^{(1,0)}\tilde{\Q}_{\tilde{X}}$ of vector fields. 

\begin{definition} 
A cohomology F-manifold is a differential graded manifold $(\tilde{\Q},{\tilde{X}})$ equipped with a graded commutative and associative product for vector fields $$\tilde{\star}: \T^{(1,0)}\tilde{\Q}_{\tilde{X}}\otimes\T^{(1,0)}\tilde{\Q}_{\tilde{X}}\to\T^{(1,0)}\tilde{\Q}_{\tilde{X}}.$$ 
\end{definition}

Note that, in his papers, Merkulov is working with different definitions of cohomology F-manifolds and $\operatorname{F}_{\infty}$-manifolds, a generalization of cohomology F-manifold making use of the higher homotopies of the product. Since in \cite{M1} Merkulov does not ask for any kind of integrability (in the spirit of Hertling-Manin's F-manifolds) and in \cite{M2} no longer requires the existence of Euler and unit vector fields, we have decided to leave out all these extra requirements in the paper. Furthermore note that Merkulov allows $\tilde{\Q}$ to be any formal pointed graded manifold, which is more general in the sense that each (graded) vector space naturally carries the structure of a formal pointed (graded) manifold with the special point being the origin, see \cite{K}. Finally we remark that our cohomology F-manifolds are indeed infinite-dimensional and formal in the sense that we do not specify a topology on them. \\

With this we can now state 

\begin{theorem}
The big pair-of-pants product equips the differential graded manifold from section two with the structure of a cohomology F-manifold in such a way that, at the tangent space at zero, we recover the (small) pair-of-pants product on cylindrical contact homology. 
\end{theorem}

 \begin{proof} The proof splits up into three parts. \\

\noindent\emph{Vector field product:} Obviously the main ingredient for the proof is proposition \ref{master}. There we have shown that big pair-of-pants product $\tilde{\star} \in \T^{(1,2)}\tilde{\Q}$ and the cohomological vector field $\tilde{X}\in\T^{(1,0)}\tilde{\Q}$ from full contact homology satisfy the master equation $\LL_{\tilde{X}} \tilde{\star} = 0$. This should be seen as generalization of the master equation relating the small pair-of-pants product and the boundary operator of Floer homology, that is, cylindrical contact homology. In the same way as the latter proves that the small pair-of-pants product descends to a product on cylindrical contact homology, the master equation of proposition \ref{master} shows that the big pair-of-pants product defines an element in the tensor homology $H_*(\T^{(1,2)}\tilde{\Q},\LL_{\tilde{X}})$. Following the discussion of the concept of differential graded manifolds in section two, it follows from the compatibility of the Lie derivative with the contraction of tensors that the big pair-of-pants product defines a $(1,2)$-tensor field $\tilde{\star}\in\T^{(1,2)}\tilde{\Q}_{\tilde{X}}$, that is, a product of vector fields on the differential graded manifold $(\tilde{\Q},\tilde{X})$. \\
 
\noindent\emph{Commutativity:} In order to prove the main theorem, it remains to show that this product is commutative and associative (in the graded sense). In order to see that the big pair-of-pants product is already commutative \emph{on the chain level}, recall that we have already assumed that the position of the constrained additional marked point is chosen to be invariant under reordering of the ordered tuple $(z_0,z_1)$ (and, obviously, also of $(z_2,\ldots,z_{r-1})$). Hence it directly follows that the count of elements in $\IM^{\gamma^+}_{\gamma_0,\gamma_1}(\Gamma)$ and  $\IM^{\gamma^+}_{\gamma_1,\gamma_0}(\Gamma)$ agree and hence also  $$\frac{\del}{\del q_{\gamma_0}}\;\tilde{\star}\; \frac{\del}{\del q_{\gamma_1}}\;=\; \frac{\del}{\del q_{\gamma_1}}\;\tilde{\star}\; \frac{\del}{\del q_{\gamma_0}}$$
on the chain level, up to a sign determined by the $\IZ_2$-grading of the formal variables $q_{\gamma_0}$, $q_{\gamma_1}$. \\

\noindent\emph{Associativity:} In contrast to commutativity, we cannot expect to have associativity for the big pair-of-pants on the chain level, but only after viewing it as a vector field product on the differential graded manifold $(\tilde{\Q},\tilde{X})$. More precisely, we show that the resulting $(1,3)$-tensors $\tilde{\star}_{10,2}$ and $\tilde{\star}_{1,02}$ on $\tilde{\Q}$ defined by 
\begin{eqnarray*} 
\tilde{\star}_{10,2}:\;\frac{\del}{\del q_{\gamma_0}}\otimes\frac{\del}{\del q_{\gamma_1}}\otimes\frac{\del}{\del q_{\gamma_2}}\;\mapsto\;\Big(\frac{\del}{\del q_{\gamma_1}}\;\tilde{\star}\;\frac{\del}{\del q_{\gamma_0}}\Big)\;\tilde{\star}\;\frac{\del}{\del q_{\gamma_2}},\\
\tilde{\star}_{1,02}:\;\frac{\del}{\del q_{\gamma_0}}\otimes\frac{\del}{\del q_{\gamma_1}}\otimes\frac{\del}{\del q_{\gamma_2}} \;\mapsto\;\frac{\del}{\del q_{\gamma_1}}\;\tilde{\star}\;\Big(\frac{\del}{\del q_{\gamma_0}}\;\tilde{\star}\;\frac{\del}{\del q_{\gamma_2}}\Big), \end{eqnarray*} 
do not agree, but only up to some $\LL_{\tilde{X}}$-exact term which only vanishes after passing to the differential graded manifold $(\tilde{\Q},\tilde{X})$. Due to its importance, this is the content of the following lemma, which finishes the proof of the main theorem. \end{proof} 

\begin{lemma} We have $\tilde{\star}_{10,2}-\tilde{\star}_{1,02}\;=\;\LL_{\tilde{X}}\alpha$ with some tensor field $\alpha\in\T^{(1,3)}\tilde{\Q}$. \end{lemma}

\begin{proof} The corresponding $(1,3)$-tensor field $\alpha$ is again defined by counting certain submoduli spaces $\IM^{\gamma^+}_{\gamma_0,\gamma_1,\gamma_2}(\Gamma)\subset\IM^{\gamma^+}_2(\gamma_0,\gamma_1,\gamma_2,\Gamma)$ of Floer solutions,
$$\alpha\;=\; \sum \frac{1}{\kappa_{\gamma_0}\kappa_{\gamma_1}\kappa_{\gamma_2}\kappa^{\Gamma}}\frac{1}{k}\#\IM^{\gamma^+}_{\gamma_0,\gamma_1,\gamma_2}(\Gamma,A)\cdot\; q^{\Gamma}t^{c_1(A)} \;dq_{\gamma_0}\otimes dq_{\gamma_1}\otimes dq_{\gamma_2}\otimes \frac{\del}{\del q_{\gamma^+}}.$$ Here $\IM^{\gamma^+}_{\gamma_0,\gamma_1,\gamma_2}(\Gamma)$ denotes the moduli space of $\tilde{J}$-holomorphic curves equipped with \emph{two} constrained additional marked points $z^*_1,z^*_2$, which are required to get mapped to zero under $h_2$. Here the position of the first additional marked point $z_1^*$ is fixed by the punctures corresponding to the orbits $\gamma_0$, $\gamma_1$ and $\gamma^+$, whereas the second additional marked point $z_2^*$ is fixed by the punctures corresponding to the orbits $\gamma_0$, $\gamma_2$ and $\gamma^+$. \\

Now observe that the codimension-one boundary of the moduli space $\IM^{\gamma^+}_2(\gamma_0,\gamma_1,\gamma_2,\Gamma)$ consists of components of the form 
\begin{equation}\IM^{\gamma^+}_{1,0}(\Gamma')\times\IM^{\gamma'}_{0,1}(\Gamma'')\end{equation} and \begin{equation}\IM^{\gamma^+}_{0,1}(\Gamma')\times\IM^{\gamma'}_{1,0}(\Gamma'')\end{equation} as well as of components of the form \begin{equation}\IM^{\gamma^+}_2(\Gamma')\times\IM^{\gamma'}(\Gamma'')\end{equation} and \begin{equation}\IM^{\gamma^+}(\Gamma')\times\IM^{\gamma'}_2(\Gamma'').\end{equation} Note that the difference between boundary components of type $(1)$ and those of type $(2)$ lies in the fact that in $(1)$ the first marked point lies in the upper component (and the second marked point on the lower component), while in $(2)$ the second marked points sits on the upper component (and the first marked point on the lower component). While in $(3)$ it follows that we only need to consider those components where $\Gamma'$ contains $\gamma_0$, or $\gamma_1$ and $\gamma_2$, in $(4)$ we only need to consider those components where $\Gamma''$ contains $\gamma_0$, $\gamma_1$ and $\gamma_2$. On the other hand, in $(1)$ it follows that we only need to consider the case when $\Gamma'$ contains $\gamma_1$ (and hence $\Gamma''$ contains $\gamma_0$ and $\gamma_2$), while in $(2)$ we need that $\Gamma'$ contains $\gamma_2$ (and hence $\Gamma''$ contains $\gamma_0$ and $\gamma_1$). \\

While for boundary components of type $(3)$ and $(4)$ we are again dealing with moduli spaces with two additional marked points, note that in $(1)$ and $(2)$ we are dealing with moduli spaces with one additional marked point. Arguing as in the proof of the proposition \ref{master}, it then follows that by counting the holomorphic curves in codimension-one boundary of each moduli space $\IM^{\gamma^+}_{\gamma_0,\gamma_1,\gamma_2}(\Gamma)$, we obtain the master equation of the statement. Indeed, while the boundary components of type $(1)$ and $(2)$ lead to the appearance of $\tilde{\star}_{10,2}$ and $\tilde{\star}_{1,02}$, the remaining boundary components of the type $(3)$ and $(4)$ show that associativity does not hold on the chain level but only up to the exact term $\LL_{\tilde{X}}\alpha$.  
\end{proof}

\noindent\emph{Invariance}\\

We now turn to the invariance properties of the new objects. For this we show that for different choices of auxiliary data like (domain-dependent) almost complex structures and Hamiltonians, we obtain cohomology F-manifolds which are isomorphic in the natural sense. \\

Let $(\tilde{\Q}^+,\tilde{X}^+,\tilde{\star}^+)$ and $(\tilde{\Q}^-,\tilde{X}^-,\tilde{\star}^-)$ be pairs of differential graded manifolds equipped with vector field products $\tilde{\star}^{\pm}$ in the above sense, obtained using two different choices of cylindrical almost complex structures $\tilde{J}^{\pm}$ on $\RS\times M$ defined using two different choices of (domain-dependent) Hamiltonian functions $H^{\pm}$ and $\omega$-compatible almost complex structures $J^{\pm}$. Recall from section two that, by choosing a smooth family $(H_s,J_s)$ of (domain-dependent) Hamiltonians and $\omega$-compatible almost complex structures interpolating between $(H^+,J^+)$ and $(H^-,J^-)$, we can equip the cylindrical manifold $\RS\times M$  with the structure of an almost complex manifold with cylindrical ends in the sense of \cite{BEHWZ}, where we denote the resulting (non-cylindrical) almost complex structure on $\RS\times M$ by $\hat{J}$. \\

Recall further that, by counting elements in moduli spaces $\widehat{\IM}^{\gamma^+}(\Gamma)=\widehat{\IM}^{\gamma^+}(\Gamma,A)$ of $\hat{J}$-holomorphic curves, we get a morphism of the underlying differential graded manifolds $(\tilde{\Q}^+,\tilde{X}^+)$ and $(\tilde{\Q}^-,\tilde{X}^-)$. It is given by a linear map $\varphi^{(0,0)}:\T^{(0,0)}\tilde{\Q}^+\to\T^{(0,0)}\tilde{\Q}^-$ defined as $$\varphi^{(0,0)}(q_{\gamma^+})\;=\;\sum_{\Gamma,A}\frac{1}{\kappa^{\Gamma}}\cdot\#\widehat{\IM}^{\gamma^+}(\Gamma,A)\cdot q^{\Gamma} t^{c_1(A)}$$ and $\varphi^{(0,0)}(q_{\gamma^+_1}\cdot\ldots\cdot q_{\gamma^+_s}):=\varphi^{(0,0)}(q_{\gamma^+_1})\cdot\ldots\cdot\varphi^{(0,0)}(q_{\gamma^+_s})$. In particular, the map $\varphi^{(0,0)}$ on the space of functions on $\tilde{\Q}$ uniquely defines the corresponding map $\varphi^{(0,1)}:\T^{(0,1)}\tilde{\Q}^+\to\T^{(0,1)}\tilde{\Q}^-$ on the space of forms by the requirement that $d\circ\varphi^{(0,0)}=\varphi^{(0,1)}\circ d$ with the exterior derivative $d:\T^{(0,0)}\tilde{\Q}^{\pm}\to\T^{(0,1)}\tilde{\Q}^{\pm}$. \\

While in \cite{EGH} it is shown $\varphi^{(0,0)}$ indeed defines an isomorphism of the space of functions, $\varphi^{(0,0)}:\T^{(0,0)}\tilde{\Q}^+_{\tilde{X}^+}\stackrel{\cong}{\longrightarrow}\T^{(0,0)}\tilde{\Q}^-_{\tilde{X}^-},$ the map on the spaces of forms on the differential graded manifolds, $\varphi^{(0,1)}: \T^{(0,1)}\tilde{\Q}^+_{\tilde{X}^+}\to\T^{(0,1)}\tilde{\Q}^-_{\tilde{X}^-}$, is also an isomorphism by the same arguments, see also the section on satellites in \cite{EGH}. On the other hand, having shown that the spaces of functions and one-forms on $(\tilde{\Q}^{\pm},\tilde{X}^{\pm})$ are isomorphic, this automatically implies that the same is true for all the other tensor fields $\T^{(r,s)}\tilde{\Q}^{\pm}_{\tilde{X}^{\pm}}$. 

\begin{theorem} Let $M$ be a symplectic manifold with contact-type boundary and assume that all Hamiltonian functions have fixed asymptotic linear slope in the cylindrical end of the completion. Then for different choices of auxiliary data like (domain-dependent) Hamiltonian functions $H^{\pm}$ and $\omega$-compatible almost complex structures $J^{\pm}$, the isomorphism between the resulting differential graded manifolds $(\tilde{\Q}^+,\tilde{X}^+)$ and $(\tilde{\Q}^-,\tilde{X}^-)$ extends to an isomorphism of cohomology F-manifolds. In the case when the symplectic manifold is closed, we recover the cohomology F-manifold structure on $\Lambda^+\QH^*(M)=\bigoplus_{k\in\IN}\QH^*(M)$ given by the big quantum product. \end{theorem}

\begin{proof} We prove that the isomorphism of tensor fields respects the product structure by showing that $$\varphi^{(0,2)}\circ\tilde{\star}=\tilde{\star}\circ\varphi^{(0,1)}$$ where the isomorphism $\varphi^{(0,2)}:\T^{0,2)}\tilde{\Q}^+_{\tilde{X}^+}\to\T^{(0,2)}\tilde{\Q}^-_{\tilde{X}^-}$ is on the chain level given by $\varphi^{(0,2)}(q^{\Gamma_+}\cdot dq_{\gamma^+_1}\otimes dq_{\gamma^+_2}) = \varphi^{(0,0)}(q^{\Gamma_+})\cdot \varphi^{(0,1)}(dq_{\gamma^+_1})\otimes\varphi^{(0,1)}(dq_{\gamma^+_2})$. Denote by $\widehat{\IM}^{\gamma^+}_1(\gamma_0,\gamma_1,\Gamma)$ the moduli space of $\hat{J}$-holomorphic curves $(\tilde{u},z_0,z_1,z_2,\ldots,z_{r-1},z^*)$ in the topologically trivial symplectic cobordism $\RS\times M$ interpolating between the two choices of auxiliary data $(H^{\pm},J^{\pm})$ and with an additional unconstrained marked point. \\

As in the cylindrical case, there still exists an evaluation map to the circle, $\ev: \widehat{\IM}^{\gamma^+}_1(\gamma_0,\gamma_1,\Gamma)\to S^1$ given by mapping $(\tilde{u},0,1,z_2,\ldots,z_{r-1},z_{\infty},z^*)$ to $h_2(z^*)$ with $\tilde{u}=(h_1,h_2,u):\Si\to M$. Summarizing we can again define new moduli spaces $\widehat{\IM}^{\gamma^+}_{\gamma_0,\gamma_1}(\Gamma)\subset\widehat{\IM}^{\gamma^+}_1(\gamma_0,\gamma_1,\Gamma)$ by fixing the position of the additional marked point $z^*$ using the three special punctures and requiring that it gets mapped to zero. \\

For the proof we now have to consider the codimension one-boundary of the submoduli space $\widehat{\IM}^{\gamma^+}_{\gamma_0,\gamma_1}(\Gamma)$. Instead of two types of two-level curves we now have four types of two-level curves, where the factor two just results from the fact that we have to distinguish which level is cylindrical and which is non-cylindrical. As in the corresponding result for satellites in \cite{EGH}, we obtain that the morphism and the product commute up to terms which are exact for the Lie derivatives with respect to $\tilde{X}^{\pm}$. \\

Finally, note that in the case when $H=0$, the new moduli spaces used for the definition of the big pair-of-pants product count Floer solutions $u:\Si\to M$ satisfying the Cauchy-Riemann equation $\CR_J(u)=0$ with varying conformal structure and fixed asymptotic markers. Since each such solution indeed extends (by the removable singularity theorem) to a map from the closed sphere to $M$ and we can hence also the asymptotic markers are not needed anymore, we precisely end up with the moduli spaces of $J$-holomorphic spheres in $M$ defining the big quantum product. On the other hand, when $H=0$, the cohomological vector field $\tilde{X}$ is vanishing. For this observe that, in contrast to the moduli spaces for the big pair-of-pants product, we now count Floer solutions with (simultaneously) rotating asymptotic markers. After setting $H=0$ it follows that we still arrive at moduli spaces of $J$-holomorphic spheres but with a free $S^1$-symmetry given by the unconstrained rotating asymptotic markers. \end{proof}

In particular, up to isomorphism the cohomology F-manifold structure is contained in the Frobenius manifold structure on quantum homology given by all rational Gromov-Witten invariants and hence has already been computed in many cases, see \cite{DZ}. \\

Furthermore, extending what we already remarked in the section on differential graded manifolds, by abstract homotopy transfer we claim that there exists a cohomological vector field $X\in\T^{(1,0)}\Q$ but also a vector field product $\tilde{\star}'\in\T^{(1,2)}\Q$ on $\Q:=\HC_*^{\cyl}(M_{\phi})$ such that $(\tilde{\Q},\tilde{X},\tilde{\star})$ and $(\Q,X,\star)$ define the same cohomology F-manifold, up to homotopy. 

\section{$L_{\infty}$-structure and Reeb dynamics}

In this section we show how the newly introduced $L_{\infty}$-structure in Hamiltonian Floer theory can be used to prove the existence of (multiple) closed Reeb orbits. In order to link our new invariant to Reeb dynamics, we consider Hamiltonians which are special in the sense that they are time-independent, $C^2$-small in the interior and depending only on the $\IR$-coordinate in the cylindrical end. \\

In this case it is well-known, see e.g. \cite{A}, \cite{R}, that the one-periodic orbits of $H$ and its multiples $kH$, $k\in\IN$, correspond to critical points of $H$ in the interior or project onto closed orbits of the Reeb vector field on the contact-type boundary in case they sit in the cylindrical end. It follows that the chain space of $\HC^{\cyl}_*(M_{\phi})=\bigoplus_k\HF_*(\phi^k)$ is generated by critical points and closed Reeb orbits. More precisely, while a critical point gives a generator for $\HF_*(\phi^k)$ (on the chain level) for each $k\in\IN$, for the closed Reeb orbits the situation is slightly more complicated: First, depending on the period of the closed Reeb orbit, a closed Reeb orbit appears as a generator of $\HF_*(\phi^k)$ only if $k$ is large enough. Note that the latter is precisely the reason why the symplectic homology $\SH_*(M)$ of $M$ is defined as a direct limit of Floer homology groups, $\SH_*(M)=\lim_{k\to\infty}\HF_*(\phi^k)$, see \cite{R}. But even more important, since in the definition of Hamiltonian Floer theory we need to consider parametrized orbits, we emphasize that the choice of a time-independent Hamiltonian indeed leads to a Morse-Bott case. As described in \cite{BO}, it follows that each closed Reeb orbit indeed gives two generators for $\HF_*(\phi^k)$ for $k$ large enough, which can be obtained by viewing the closed Reeb orbitas an orbit with a fixed parametrization or as an $S^1$-family of parametrized orbits. Note that using the relation between parametrized Hamiltonian orbits in $M$ and closed orbits in the mapping torus $M_{\phi}$ discussed above, this amounts to having a fixed closed orbit or an $S^1$-family of orbits in $M_{\phi}$.\\

First we show that the results from \cite{F1} immediately lead to the proof of the following 

\begin{theorem}
If the $L_{\infty}$-structure on $\HC^{\cyl}_*(M_{\phi})=\bigoplus_k\HF_*(\phi^k)$ is not trivial, then there is at least one closed Reeb orbit on the contact-type boundary of $M$. In particular, the $L_{\infty}$-structure is trivial in the case of closed symplectic manifolds.
\end{theorem}

\begin{proof} 
We show that this can be deduced from the results in \cite{F1}. Assume that there is no closed Reeb orbit on the contact-type boundary. Then it follows that all one-periodic Hamiltonian orbits correspond to critical points. Furthermore, we can directly assume, by the maximum principle, that the holomorphic curves stay in the part of $M$ where $H$ is still $C^2$-small. In the same way as for closed symplectic manifolds it can then be shown that, after passing to homology, the corresponding cohomological vector field $\tilde{X}$ and hence the $L_{\infty}$-structure is indeed trivial. This follows from the fact that all moduli spaces of holomorphic curves with three or more punctures then come with an $S^1$-symmetry as proven in \cite{F1}. Note that in this case all periodic orbits correspond to critical points, so that the simultaneously rotating asymptotic markers are unconstrained. \\

We emphasize that the domain-dependent Hamiltonian perturbations defined in \cite{F1} arise as special case of the domain-dependent Hamiltonian perturbations defined in the appendix when each map $H(\vec{k}):\IM_1(\vec{k})\to C^{\infty}(M)$ is given by $H^{(r+1)}:\IM_{r+2}\to C^{\infty}(M)$ in the sense that it factors through the map $\ft: \IM_1(\vec{k})\to\IM_{r+2}$ forgetting the asymptotic markers and the multiplicities, $$H(\vec{k})=H^{(r+1)}\circ\ft:\, \IM_1(\vec{k})\to \IM_{r+2}\to C^{\infty}(M). $$
\end{proof}

As mentioned above, for each $k\in\IN$ the chain space for $\HF^*(\phi^k)$ is generated by critical points as well as closed Reeb orbits up to a maximal period. In order to be able to drop the request of the maximal period one often prefers to work with the symplectic homology $\SH_*(M)$ of $M$ which is defined as a direct limit of Floer homology groups, $\SH_*(M)=\lim_{k\to\infty}\HF_*(\phi^k)$, see \cite{R} for details on the construction of the direct limit. As an immediate consequence we find

\begin{corollary}
If the Lie bracket on the symplectic homology of $M$ does not vanish, then there is at least one closed Reeb orbit on the contact-type boundary of $M$.
\end{corollary}

Indeed it suffices to observe that the nontriviality of the Lie bracket on $\SH_*(M)$ immediately implies that, for $k,\ell$ large enough, the restricted Lie bracket $m_2^{k,\ell}:\HF_*(\phi^k)\times\HF_*(\phi^{\ell})\to\HF_*(\phi^{k+\ell})$ has to be nontrivial.  \\
 
Going beyond, we now show how the results of \cite{F2} can be used to prove the existence of multiple simple Reeb orbits. For this we first recall, see \cite{A}, \cite{R}, that there is a Batalin-Vilkovisky (BV) operator $\Delta:\SH_*(M)\to\SH_{*-1}(M)$ on symplectic homology, defined by counting Floer cylinders with unconstrained asymptotic markers. Its key property for our purposes lies in the fact that it maps parametrized Reeb orbits to their unparametrized versions and maps generators corresponding to critical points and unparametrized Reeb orbits to zero. In order to see that the definition in \cite{R} and \cite{A} immediately leads to the definition of a BV operator on $\HC^{\cyl}_*(M_{\phi})=\bigoplus_k\HF_*(\phi^k)$, it suffices to observe that the BV operator on $\SH_*(M)$ is explicitly defined through maps from $\HF_*(\phi^k)$ to $\HF_{*-1}(\phi^k)$. 

\begin{center}
\includegraphics[height=4cm]{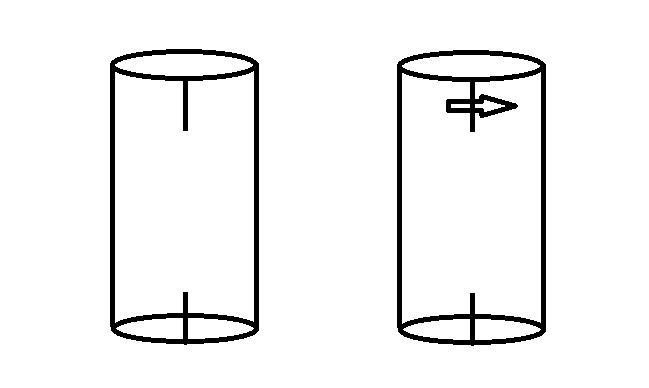}\\
\small{Boundary operator $\del$ (left) and BV operator $\Delta$ (right) in comparison}
\end{center}

\begin{theorem}
Assume that $M$ is a Liouville manifold. If the $L_{\infty}$-structure is not trivial on $\Ker\Delta\subset\HC^{\cyl}_*(M_{\phi})=\bigoplus_k\HF_*(\phi^k)$, then there either exist two simple closed Reeb orbits or one homologically trivial Reeb orbit on the contact-type boundary of $M$.
\end{theorem}

\begin{proof}
Here we show that this follows from \cite{F1} combined with the results from \cite{F2} using a proof by contradiction. Indeed let us assume from now on to the contrary that there exists only one simple closed Reeb orbit on the contact-type boundary which represents a non-trivial class in $H_1(M)$ (with $\IR$-coefficients). Assuming that we have chosen a Hamiltonian $H:M\to\IR$ of the special form above, that is, which is  which is still time-independent, $C^2$-small in the interior and depending only on the $\IR$-coordinate in the cylindrical end, note that we are in the Morse-Bott case now. By studying all the moduli spaces that a priori could contribute to the $L_{\infty}$-structure, we show that the $L_{\infty}$-structure still has to vanish on $\Ker\Delta\subset\bigoplus_k\HF_*(\phi^k)$. \\

First, it follows from the theorem 3.1 above that, if a moduli space $\IM^{\gamma}(\Gamma)$ gives a non-zero contribution, then at least one of the orbits in $\{\gamma\}\cup\Gamma$ must project to a $k$-fold covering of the simple closed Reeb orbit on the contact-type boundary of $M$. Moreover, if the orbits $\gamma_1,\ldots,\gamma_s$ in $\Gamma$ project onto the $k_0$-,...,$k_{r-1}$-fold iterates of the closed Reeb orbit, then it follows from homological reasons that indeed $k=k_0+\ldots+k_{r-1}$. Note that, in order to explicitly include critical points in the interior of $M$, here we make the convention that $k_i=0$ if $\gamma_i$ corresponds to a critical point. \\

Let us assume that $M$ is a Liouville manifold, in particular, $\omega$ is exact, and that $H=0$ in the interior of $M$. Using the relation between action of orbits and the $\omega$-energy of holomorphic curves in $\IR\times M_{\phi}$ (in the sense of \cite{BEHWZ}, 5.3) from \cite{F1}, we get that, in the Morse-Bott limit where $H$ vanishes in the interior of $M$, all holomorphic curves in $\IM^{\gamma}(\Gamma)$ have trivial $\omega$-energy. As in \cite{F2} we can deduce from (\cite{BEHWZ}, lemma 5.4) that they have image entirely contained in the image of a single orbit. In other words, $\IM^{\gamma}(\Gamma)$ entirely consists of holomorphic curves which are branched covers of the trivial cylinder over one of the periodic orbits in $M_{\phi}$ in the $S^1$-family of parametrized orbits given by the closed Reeb orbit. \\

Since we are in the Morse-Bott case now, note that, by slight abuse of notation, here $\gamma$ and the orbits $\gamma_1,\ldots,\gamma_s$ in $\Gamma$ are possibly viewed as $S^1$-family of closed orbits in $M_{\phi}$; in the notation of \cite{BO} this means that $\gamma$ can stand for $\overline{\gamma}$ or $\underline{\gamma}$ (unless it corresponds to a critical point). In our simple situation with just one simple Reeb orbit, it directly follows from the definition of the BV-operator $\Delta$ in \cite{A} that the kernel of $\Delta$ is spanned, apart from the generators corresponding to critical points, by the generators corresponding to the $S^1$-families of orbits in $M_{\phi}$. \\

It follows that, after restricting to $\Ker\Delta\subset\HC^{\cyl}_*(M_{\phi})$, the $L_{\infty}$-structure only gets contributions from orbit curves in the sense of \cite{F2}. Indeed, by viewing everything in terms of Floer curves in $M$, recall from proposition 2.7 that the coefficients of the $L_{\infty}$-structure count Floer curves with simultaneously rotating asymptotic markers. When restricted to $\Ker\Delta$ it follows that the single orbit in the positive end must correspond to a fixed orbit, since else the rotating asymptotic markers could again not be fixed which again would lead to an $S^1$-symmetry on the moduli space. On the other hand, in the latter case we can indeed forget about the rotating asymptotic markers, that is, we just need to count branched covers of orbit cylinders as in \cite{F2}. \\

Note that at the moment we assume that the Hamiltonian $H$ is zero in the interior of $M$; in particular, if $k_i=0$, then the corresponding puncture is an unconstrained additional marked point in the sense of \cite{EGH}. Instead of perturbing $H$ to be a $C^2$-small Morse function, note that we can use the additional marked points to define evaluation maps which in turn can be used to pullback de Rham homology classes on $M$ in order to integrate them over the moduli space. As shown in \cite{F2}, by dimension reasons the only possible contributions from orbit curves with one or more additional marked points come from orbit cylinders with one additional marked, where the corresponding coefficient is given by pairing of the homology class of degree one with the homology class in $H_1(M)$ represented by the closed Reeb orbit. Since the critical points are generators for the relative cohomology classes in $H^*(M,\del M)$ and the orbit indeed represents a class in $H_1(\del M)$, it follows that this pairing indeed gives zero. \\

It follows that, for $H$ sufficiently $C^2$-small in the interior of $M$, the only possible contributions to the $L_{\infty}$-structure restricted to $\Ker\Delta$ come from orbit curves without additional marked points, i.e., where all orbits in $\{\gamma\}\cup\Gamma$ indeed correspond to iterates of the closed Reeb orbit on the contact-type boundary of $M$. While for orbit cylinders the Fredholm index is always zero, there indeed exist examples of moduli spaces of branched covers which by index reasons could possibly contribute to the $L_{\infty}$-structure. On the other hand, we have shown in \cite{F2} that, after perturbing the Cauchy-Riemann operator in order to achieve regularity, the count of elements in the perturbed moduli spaces indeed gives zero - in contradiction to our assumption that the $L_{\infty}$-structure is nontrivial on $\Ker\Delta$. \\

In order to recall the main ideas from \cite{F2}, note that the moduli space is a complex manifold of complex dimension greater or equal to one so that, when the Fredholm index is assumed to be one, the actual dimension of the moduli space must be strictly larger than its virtual dimension expected by the Fredholm index. Note that this in turn implies that the moduli space cannot be transversally cut out by the Cauchy-Riemann operator, even for generic choices of $J$. In \cite{F2} it is shown that for orbit curves the cokernels of the linearized Cauchy-Riemann operator fit together to a smooth obstruction vector bundle $\Coker\CR_J$. By induction on the dimension of the unperturbed moduli space of branched covers it is shown in \cite{F2} that for every pair of coherent and transversal sections in $\Coker\CR_J$ the algebraic count of zeroes agree and that the resulting Euler number vanishes, $\chi(\Coker\CR_J) = 0.$ \\

In order to finish the proof of the theorem by contradiction, it just remains to comment on the perturbation schemes used in order to prove regularity for the appearing moduli space. \\

While for the general definition of the $L_{\infty}$-structure (as well as for the moduli spaces with orbits corresponding to critical points) we have used domain-dependent Hamiltonian perturbations, we have used (as in \cite{F2}) obstruction bundle sections in order to get regularity for the moduli spaces of orbit curves; in particular, we assume that the Hamiltonian is unperturbed in the cylindrical end of $M$. By choosing domain-dependent Hamiltonian perturbations which coherently connect the two given special choices in the sense of the corresponding definitions in \cite{F1}, \cite{F2} and assuming that for the cobordism the obstruction bundle section is again extended to a neighborhood of the orbit similar as in the proof of the proposition 3.1 of \cite{F2}, it follows from standard arguments that the resulting invariants are still the same. In particular, as in proposition 4.1 in \cite{F2} it follows that we can still use that the orbit curves do not contribute. Note that for the necessary monotonicity assumption we can possibly replace the appearing Hamiltonians by a suitable multiple. Finally note that in the case when we claimed to get a $S^1$-symmetry on the moduli space of orbit curves, the underlying moduli space of branched covers gets an extra $S^1$-factor (corresponding to the fact that the rotating asymptotic markers are not fixed). The  extra $S^1$-action now directly lifts an $S^1$-action on the obstruction bundle, so choosing an $S^1$-equivariant section in it gives a regular moduli space which still carries this $S^1$-symmetry; in particular, the Morse-Bott situation does not cause additional regularity problems.       
\end{proof}
 
\begin{remark} 
The proof of our theorem indeed shows that there must exist closed Reeb orbits on the contact-type boundary of $M$ such that their first homology classes are linearly dependent in $H_1(M)$. 
\end{remark}

As above there also exists an immediate corollary which just makes use of the Lie bracket and the BV operator on symplectic homology.
\begin{corollary}
If there exist $\alpha,\beta\in\SH_*(M)$ with $\Delta(\alpha)=0=\Delta(\beta)$ and $[\alpha,\beta]\neq 0$, then there either exist two simple closed Reeb orbits or one homologically trivial Reeb orbit on the contact-type boundary of $M$.
\end{corollary}

\section{Relation with mirror symmetry}

Assume that $M$ and $M^{\vee}$ are open Calabi-Yau manifolds which are mirror to each other in the sense of homological mirror symmetry.  While in the case of closed Calabi-Yau manifolds $M$ and $M^{\vee}$ the corresponding closed-string version relates the Gromov-Witten theory of $M$ with the (extended) deformation theory of complex structures on $M^{\vee}$, passing from closed to open manifolds, Gromov-Witten theory has to be replaced by symplectic homology, see \cite{Pas}. More precisely, it is conjectured that there exists a linear isomorphism between the symplectic homology $\SH_*(M)$ and the Dolbeault cohomology $\H^*(M^{\vee},\bigwedge^* T_{M^{\vee}})$ of polyvector fields which intertwines the natural BV operators and Lie brackets on both sides.\\

Before we can explain below how mirror symmetry combined with our results from the last section can be used to predict the existence of (multiple) closed Reeb orbits, we first quickly recall the basics of deformation theory of complex manifolds, where we refer to \cite{H} for a detailed exposition. \\

Note that a choice of an almost complex structure $J$ on $M^{\vee}$ defines a unique splitting of the complexified tangent bundle into $(1,0)$- and $(0,1)$-part, $$TM^{\vee}\otimes\IC=T^{1,0}\oplus T^{0,1}.$$ On the other hand, a complex structure $J$ can be characterized by the fact that the $(0,1)$-part of the complexified tangent bundle is closed under the Lie bracket, $[T^{0,1},T^{0,1}]\subset T^{0,1}$. While it easily follows that, near the fixed complex structure $J^{\vee}$ on $M^{\vee}$, the space of almost complex structures can be identified with the space $\IA^{(0,1)}(T_{M^{\vee}})$ of $(0,1)$-forms with values in the holomorphic tangent bundle $T_{M^{\vee}}=T^{1,0}$ of $M^{\vee}=(M^{\vee},J^{\vee})$, the subset of complex structures agrees locally with the solution set of the Maurer-Cartan equation $$\CR q \;+\;\frac{1}{2}[q,q]\,=\,0,\,\,q\in\IA^{(0,1)}(T_{M^{\vee}}).$$ Here $\CR:\IA^{(0,q)}(T_{M^{\vee}})\to \IA^{(0,q+1)}(T_{M^{\vee}})$ is the Dolbeault operator and $[\cdot,\cdot]: \IA^{(0,p)}(T_{M^{\vee}})\otimes\IA^{(0,q)}(T_{M^{\vee}})\to \IA^{(0,p+q)}(T_{M^{\vee}})$ is the natural Lie bracket on polyvector-valued forms, see \cite{BK}. \\

In order to arrive at the algebraic invariants relevant for mirror symmetry, we slightly need to generalize the above setup and study the extended deformation theory of complex structures, which is now locally modelled by the dg Lie algebra of polyvector-valued forms $(\IA^{(0,*)}(\bigwedge^* T_{M^{\vee}}),\CR,[\cdot,\cdot])$. While the Dolbeault cohomology of polyvector fields is now defined as the homology of $\CR$, $$\H^*(M^{\vee},\bigwedge^* T_{M^{\vee}})\,:=\,H_*(\IA^{(0,*)}(\bigwedge^* T_{M^{\vee}}),\CR),$$ the Lie bracket on $\IA^{(0,*)}(\bigwedge^* T_{M^{\vee}})$ immediately descends to a Lie bracket $[\cdot,\cdot]$ on $H_*(M^{\vee},\bigwedge^* T_{M^{\vee}})$. Finally, in order to introduce a natural BV-operator on $H_*(M^{\vee},\bigwedge^* T_{M^{\vee}})$, observe that the classical $\del$-operator on the space $\IA^{(p,q)}(M^{\vee})$ of $(p,q)$-forms on $M^{\vee}$ defines an operator $\Delta$ on $\IA^{(0,q)}(\bigwedge^p T_{M^{\vee}})$ using the isomorphism $\IA^{(0,q)}(\bigwedge^p T_{M^{\vee}})\cong\IA^{(n-p,q)}(M^{\vee})$ given by a holomorphic volume form $\Omega\in\IA^{(n,0)}(M^{\vee})$. Note that this is the point where we explicitly use that $M^{\vee}$ is assumed to be Calabi-Yau. Then the compatibility of the $\del$-operator with the $\CR$-operator implies that $\Delta$ defines a BV-operator on $\H^*(M^{\vee},\bigwedge^*T_{M^{\vee}})$. \\

As mentioned above, if $M$ and $M^{\vee}$ are mirror to each other, then $\SH_*(M)$ is supposed to be isomorphic to  $\H^*(M^{\vee},\bigwedge^*T_{M^{\vee}})$ and the isomorphism is supposed to be compatible with the Lie brackets and BV-operators on both sides. This said, we get the following corollary to our results from the last section.

\begin{corollary}
If the Lie bracket on $\H^*(M^{\vee},\bigwedge^* T_{M^{\vee}})$ does not vanish, then there is at least one closed Reeb orbit on the contact-type boundary of $M$. If there exist $\alpha,\beta\in\H^*(M^{\vee},\bigwedge^* T_{M^{\vee}})$ with $\Delta(\alpha)=0=\Delta(\beta)$ and $[\alpha,\beta]\neq 0$, then there either exist two simple closed Reeb orbits or one homologically trivial Reeb orbit on the contact-type boundary of $M$.
\end{corollary}

\section{Applications to quantum cohomology of the quintic}

In this section we show how our newly defined big pair-of-pants product and corollary 3.3 can be used to give a precise formulation of a conjecture due to Seidel. In the first section in \cite{Se} he sketches a way for computing the small quantum product of a complex projective variety by counting Floer curves with arbitary many punctures in a suitable divisor complement. More precisely, let $\overline{M}=(\overline{M},\omega)$ be a smooth complex projective variety with trivial canonical bundle and $D\subset\overline{M}$ a smooth hyperplane section in it whose homology class is Poincare dual to $\omega$. Then the complement $M:=\overline{M}\backslash D$ is an exact symplectic manifold with cylindrical end $\IR^+\times V$ over the unit normal bundle $S^1\to V\to D$ of $D$ in $\overline{M}$ with trivial first Chern class. \\

Note that the small quantum cohomology $\QH^*(\overline{M})$ is given as the tensor product of singular cohomology $\H^*(\overline{M})$ with the Novikov ring $\Lambda^0$ of power series in a formal variable $t$ of degree zero. Let $\star^k$ denote the part of the small quantum product defined by counting holomorphic spheres which intersect the divisor $D$ $k$-times. Defining the small quantum product on $\QH^*(\overline{M})$ as the formal sum $\star =\sum_k \star^k t^k$, note that it can also be viewed as a family of deformed products $\star_t:\H^*(\overline{M})\otimes\H^*(\overline{M})\to\H^*(\overline{M})$. Restricting further to the case where $\overline{M}$ denotes the hypersurface of degree $5$ in $\mathbb{C}\mathbb{P}^4$, it is known that the small quantum product is only nontrivial in degree two, that is, it restricts to a family of products $\star_t: \H^2(\overline{M})\otimes\H^2(\overline{M})\to\H^4(\overline{M})$, where it actually counts isolated rational curves (and their multiple covers) for a generic choice of compatible almost complex structure. Note that $\H^2(\overline{M})$ is generated by $\omega$ and that $\H^4(\overline{M})$ is actually isomorphic to $\H^2(\overline{M})$ via Poincare duality by mapping $\omega\wedge\omega$ to $\omega$. \\

On the other hand, as outlined in \cite{Se}, it follows from the hard Lefschetz theorem that the symplectic cohomology group $\SH^2(M)$ of the exact symplectic manifold $M$ is also one-dimensional. Let $H:M\to\IR$ be an asymptotically linear Hamiltonian, where we now explicitly assume that $H$ depends only on the $\IR$-coordinate in the cylindrical end and is equal to zero away from the cylindrical end; in particular, $H$ is of Morse-Bott type. Assuming that the linear slope is sufficiently large, it follows that $\HF^0(\phi^k)$ is isomorphic to $\SH^2(M)$ for all $k\in\IN$, where the canonical generator of $\HF^0(\phi^k)$, again denoted by $\omega$, can be represented by the Morse-Bott family of closed Reeb orbits of period one given by all fibres of the $S^1$-bundle $S^1\to V\to D$. For the difference between the two gradings, note that in this paper we use the new degree convention motivated by \cite{EGH}. \\

After using the formal variable $t$ in order to write a general element in $\HF^0(\phi)$ as $q=t\omega$, we claim that it follows from an obvious $S^1$-symmetry of the relevant moduli spaces that we indeed have $X(q)=0$. Since obviously also $E(q)=0$ for the Euler vector field, it follows from corollary 3.3 that we obtain a family of deformed pair-of-pants products $$\star_q: \bigoplus_k\HF^0(\phi^k)\otimes\bigoplus_k\HF^0(\phi^k)\to\bigoplus_k\HF^2(\phi^k).$$ Note that we actually claim (without proof) that the deformed products still live on the same Floer cohomology groups. On the other hand, note the family of deformed products can again be expanded as $\star_q=\sum_k \star^k q^k$, where $\star^k:\HF^0(\phi^{k_0})\otimes\HF^0(\phi^{k_1})\to\HF^2(\phi^{k_0+k_1+k})$ is defined by counting Floer curves with $k+3$ cylindrical ends.\\

With this we are ready for our attempt to formalize the conjecture sketched in \cite{Se}. Let $\omega\wedge\omega$ denote the generator of $\HF^2(\phi^{k_0+k_1+k})$ represented by the Morse-Bott family of closed Reeb orbits of period one given by all fibres of the $S^1$-bundle $S^1\to V\to Z$ where the subdivisor $Z\subset D$ is chosen such that it is Poincare dual to $\omega\wedge\omega\in\H^4(\overline{M})$. 

\begin{conjecture}
The coefficient of the deformed pair-of-pants product $\omega\star^k\omega$ of the two canonical generators of $\HF^0(\phi^{k_0})$ and $\HF^0(\phi^{k_1})$ in front of the generator $\omega\wedge\omega\in\HF^2(\phi^{k_0+k_1+k})$ agrees with the count holomorphic spheres of degree $k$ in the quintic hypersurface $\overline{M}$.
\end{conjecture}

In \cite{Se} it is outlined how the Fukaya category of $\overline{M}$ can be computed using the (wrapped) Fukaya category of the divisor complement $M$. In the same way as Seidel considers this as a promising approach to understand the Fukaya category of complex projective variety and hence also of open-string mirror symmetry, we claim that the above conjecture provides a novel approach to understand its closed-string version; its proof is work in progress. 

\section*{Appendix: Transversality using domain-dependent Hamiltonians}

While the required transversality for all appearing moduli spaces can be established using the polyfold theory of Hofer-Wysocki-Zehnder and has indeed already been established using an entirely different approach by J. Pardon in \cite{Par}, in this appendix we additionally show how to adapt the results of the author in \cite{F1} to establish the necessary nondegeneracy of orbits and transversality for all appearing moduli spaces. Since everything is only a mild generalization of the results from \cite{F1}, we only focus on the changes that need to be made, and refer for details to the detailed paper \cite{F1}. \\

We want to emphasize that, even after employing domain-dependent Hamiltonian perturbations, we still keep the monotonicity features for the Floer curves. For this we will assume that the Hamiltonian perturbations are in fact fixed (and hence domain-independent) outside a compact region containing the closed Hamiltonian orbits, so that a maximum principle still exists. On the other hand, the resulting class of perturbations is still large enough in order to achieve transversality via the Sard-Smale theorem, since the holomorphic curves by the maximum principle never leave this compact region.  Note that it is not possible to achieve transversality using domain-dependent almost complex structures $J$ on $M$, since the latter do not affect the orbit curves studied in \cite{F2}. \\   

For the discussion we have to distinguish between domain-stable holomorphic curves (the underlying punctured sphere is already stable in the sense that it has no nontrivial automorphisms, which means that it carries at least three punctures) and domain-unstable holomorphic curves like holomorphic spheres, holomorphic planes and holomorphic cylinders. \\

\noindent\emph{Holomorphic spheres, planes and cylinders}\\

First, it is a well-known result from Gromov-Witten theory, see \cite{MDSa}, that one can prove regularity for all appearing moduli spaces of holomorphic spheres when the underlying symplectic is semi-monotone. \\ 

As in \cite{F1} we observe next that there exist no holomorphic planes in $\IR\times M_{\phi}$, which simply follows from the fact that there is no branched covering map from the plane to the cylinder. \\

On the other hand, we have seen that the cylindrical contact homology complex of $M_{\phi}$ is precisely given by the sum of the Floer homology complexes for all powers $\phi^k$ of the Hamiltonian symplectomorphism $\phi$. It follows that the transversality problem for domain-unstable curves in SFT of (Hamiltonian) mapping tori reduces to the transversality results for symplectic Floer homology. \\

Apart from assuming monotonicity in order to be able to deal with bubbling-off of holomorphic spheres as described above, it is a classical result (see \cite{MDSa}), that one can prove nondegeneracy for all fixed points and transversality for all moduli spaces of cylinders when one considers for every $k\in\IN$ a sufficiently generic time-dependent Hamiltonian function $H=H^k: S^1\to M$. For this recall from proposition \ref{cylindrical} that the moduli space $\IM^{\gamma^+}_{\gamma^-}$ is only non-empty when $\gamma^+$ and $\gamma^-$ have the same period $k$ by homological reasons. In this case we have that $\tilde{u}=(h,u):\RS\to\IR\times M_{\phi}\cong\RS\times M$ is a $\tilde{J}$-holomorphic cylinder precisely when $h$ is a $k$-fold unbranched covering map from the cylinder onto itself and $u:\RS\to M$ satisfies the Floer equation $\CR_{J,H,k}(u)=\del_s u + J(u) \del_t u + \nabla H^k=0$ with the $1/k$-periodic Hamiltonian $H^k_t=kH_{kt}$, where $H:S^1\to M$ is the one-periodic Hamiltonian used to define the Hamiltonian symplectomorphism $\phi$. While the $1/k$-periodicity of $H^k$ leads to the existence of multiply-covered cylinders, the latter problem can easily be resolved by perturbing $H^k$ to a generic one-periodic Hamiltonian whose time-one map still agrees with $\phi^k$. Since the generic one-periodic perturbation of $H^k$ can be viewed as a $k$-valued perturbation of the original Hamiltonian $H$, it follows that the transversality for cylindrical contact homology follows from the classical transversality result for Hamiltonian Floer homology after one considers multi-valued perturbations depending on the period $k$ of the orbits. \\  

\noindent\emph{Domain-stable holomorphic curves}\\

It remains to prove transversality for holomorphic curves with three or more punctures. While these curves lead to the involved algebraic structures discussed in \cite{EGH} and in this paper, from the point of transversality they actually cause less problems (up to the compatibility problem with the choices for the other moduli spaces) than the domain-unstable holomorphic curves. Indeed, the latter are the reason why transversality is not proved for symplectic field theory in general, which in turn was the starting point for the polyfold project of Hofer, Wysocki and Zehnder. \\

Indeed it was shown in \cite{F1} that one can prove transversality for all moduli spaces of domain-stable holomorphic curves when one introduces domain-dependent Hamiltonian perturbations, generalizing the Hamiltonian perturbations used for the moduli spaces of holomorphic cylinders discussed above. Here one uses that the underlying punctured sphere has no nontrivial automorphisms, so that one can allow the Hamiltonian to depend on points of the punctured sphere, see \cite{F1} for details. Furthermore it was shown in \cite{F1} that the resulting class of perturbations is indeed large enough to prove transversality for generic choices and that all choices can be made coherent in the sense that they are compatible with compactness and gluing of moduli spaces, which also involves the moduli spaces of domain-unstable holomorphic curves. In order to have both the symmetry condition as well as regularity, we again need to consider multi-valued Hamiltonian perturbations which destroy all multiply-covered cylinders, see \cite{CMS} for the precise definitions.  \\     

While we claim that the main results carry over naturally, here is a short discussion of how the setup of \cite{F1} needs to be improved to cover the case of general Hamiltonian mapping tori. \\

First, since we now need to employ time-dependent Hamiltonians for the cylinders, we now can no longer work with the moduli space $\IM_{r+1}$ of punctured Riemann spheres. Instead we want to use that for every moduli space $\IM^{\gamma^+}(\Gamma)$ there exists a natural map to the moduli space $\IM(\vec{k})$ of holomorphic maps from $(r+1)$-punctured sphere to the cylinder, see the proof of proposition 2.2, viewed as the moduli space for full contact homology when the symplectic manifold is the point. Here $\vec{k}=(k_0,\ldots,k_{r-1})$ is the ordered set of periods of the orbits in $\Gamma$ and the map is defined by forgetting the map $u:\Si\to M$. In other words, it only remembers the conformal structure and, in contrast to the construction in \cite{F1}, also the asymptotic markers and hence the map $h$ to the cylinder (up to $\IR$-shift). \\

After introducing an unconstrained additional marked point, we obtain the corresponding universal curve $\IM_1(\vec{k})\to\IM(\vec{k})$. Generalizing the setup in \cite{F1}, we now define a domain-dependent Hamiltonian perturbation as a map $$H(\vec{k}): \IM_1(\vec{k})\to C^{\infty}(M).$$ The fibre over each point $j\in\IM(\vec{k})$ (which now stands for the conformal structure and the asymptotic markers) defines a Hamiltonian function which depends on points on the corresponding Riemann surface with cylindrical ends. \\

After extending the universal curve to the compactification of $\IM(\vec{k})$, note that the fibre is a compact Riemann surface with boundary circles. The resulting $S^1$-parametrization near each puncture will be viewed as time coordinate for the time-dependent Hamiltonian perturbation used to prove transversality for the corresponding cylinder. Note that every end automatically has a period assigned to it. Apart from the fact that the multi-valued Hamiltonian perturbations fix the domain-dependent Hamiltonian perturbations in the cylindrical ends (see \cite{F1}), we claim that the geometrical setup to define coherent domain-dependent Hamiltonians from \cite{F1} naturally extends from the classical Deligne-Mumford moduli space of punctured spheres $\IM_{r+1}$ to the new moduli space $\IM(\vec{k})$. \\

For this observe that it follows from the standard compactness result in \cite{BEHWZ} that the codimension-one boundary $\del^1\IM(\vec{k})$ of each moduli space has components of the form $\IM(\vec{k}_1)\times\IM(\vec{k}_2)$.  This in turn implies that the codimension-one boundary $\del^1\IM_1(\vec{k})$ of the universal curve has components of the form $\IM_1(\vec{k}_1)\times\IM(\vec{k}_2)$ and $\IM(\vec{k}_1)\times\IM_1(\vec{k}_2)$, depending on whether the additional marked points is on the upper or lower component. Then we require that 
\begin{eqnarray*} 
H(\vec{k})|_{\IM_1(\vec{k}_1)\times\IM(\vec{k}_2)}&=& H(\vec{k}_1)\circ\pi_1,\\
H(\vec{k})|_{\IM(\vec{k}_1)\times\IM_1(\vec{k}_2)}&=&H(\vec{k}_2)\circ\pi_2,
\end{eqnarray*} 
where $\pi_{1,2}$ denotes the projection on the first or second factor, respectively. In particular note that $\IM_1(\vec{k})$ can still be compactified to a smooth manifold with boundary and its boundary strata are moduli spaces of curves have a less number of punctures, so that the definition of coherent Hamiltonian perturbations from \cite{F1} generalizes in the obvious way. \\

We emphasize that the domain-dependent Hamiltonian perturbations defined in \cite{F1} arise as special case when each map $H(\vec{k}):\IM_1(\vec{k})\to C^{\infty}(M)$ is given by $H^{(r+1)}:\IM_{r+2}\to C^{\infty}(M)$ in the sense that it factors through the map $\ft: \IM_1(\vec{k})\to\IM_{r+2}$ forgetting the asymptotic markers and the multiplicities, $$H(\vec{k})=H^{(r+1)}\circ\ft:\, \IM_1(\vec{k})\to \IM_{r+2}\to C^{\infty}(M). $$ 
$ $\\
On the other hand, it immediately follows that the resulting class of Hamiltonian perturbations is still large enough to prove transversality for a generic choice, as can be seen easily from the proof in \cite{F1}. Note that now the universal moduli space is the zero set of the universal Cauchy-Riemann operator $$\CR:\;\BB^{\gamma^+}(\Gamma)\times\mathcal{H}(\vec{k})\to\EE^{\gamma^+}(\Gamma)$$  in the universal Banach space bundle $\EE^{\gamma^+}(\Gamma)$ over the universal Banach manifold $\BB^{\gamma^+}(\Gamma)\times\mathcal{H}(\vec{k})$, where $\mathcal{H}(\vec{k})$ is the space of maps from $\IM_1(\vec{k})$ to $C^{\infty}(M)$. Then it can be shown as in \cite{F1} that universal Cauchy-Riemann operator is surjective and hence we obtain regularity for generic choices by the Sard-Smale theorem as in the well-known transversality theorem for somewhere-injective curves. \\  

\noindent\emph{Forgetful map is a submersion}\\

By allowing the choice of domain-dependent Hamiltonian perturbation to vary in the space $\mathcal{H}=\mathcal{H}(\vec{k})$, let $\widetilde{\IM}^{\gamma^+}(\Gamma)$ denote the corresponding universal moduli space. In the last part of this appendix we want to show that the forgetful map from $\widetilde{\IM}^{\gamma^+}(\Gamma)$ to the underlying moduli space $\IM_{r+1}$ of $r+1$-punctured spheres is a submersion. Indeed we want to prove that this continues to hold for the universal moduli spaces $\widetilde{\IM}^{\gamma^+}_{\gamma_0,\gamma_1}(\Gamma)\subset \widetilde{\IM}^{\gamma^+}(\gamma_0,\gamma_1,\Gamma)$ over the submoduli spaces  $\IM^{\gamma^+}_{\gamma_0,\gamma_1}(\Gamma)\subset \IM^{\gamma^+}(\gamma_0,\gamma_1,\Gamma)$ used to define the big pair-of-pants product. While this would imply that the image of $\widetilde{\IM}^{\gamma^+}_{\gamma_0,\gamma_1}(\Gamma)$ under this forgetful map would meet every subvariety in $\IM_{r+1}$ transversally, the latter would indeed continue to hold for the moduli space $\IM^{\gamma^+}_{\gamma_0,\gamma_1}(\Gamma)$  itself for a generic choice of domain-dependent Hamiltonian perturbation by Sard's theorem. \\

In order to prove this result, it turns out to be easier to first slightly enlarge the space of perturbations. Motivated by the similar result in \cite{We}, from now on we do not only want to allow the Hamiltonian $H$ but also the compatible almost complex structure $J$ to vary. In complete analogy one can define a domain-dependent compatible almost complex structure as map $J(\vec{k}): \IM_1(\vec{k})\to \mathcal{J}(M,\omega)$, where $\mathcal{J}(M,\omega)$ denotes the space of compatible almost complex structures on $M=(M,\omega)$. In order to obtain coherent choices for all moduli spaces, note that we can proceed precisely as above. \\

For the resulting new universal moduli space, still denoted by $\widetilde{\IM}^{\gamma^+}_{\gamma_0,\gamma_1}(\Gamma)$, we do not only allow $H=H(\vec{k})$ but also $J=J(\vec{k})$ to vary in a coherent domain-dependent fashion, where we denote by $\mathcal{J}=\mathcal{J}(\vec{k})$ the corresponding space of domain-dependent almost complex structures. While just considering domain-dependent almost complex structures is not sufficient to prove transversality due to multiple covers of orbit cylinders, see \cite{F2}, we show that the forgetful map from the enlarged universal moduli space $\widetilde{\IM}^{\gamma^+}_{\gamma_0,\gamma_1}(\Gamma)$ to $\IM_{r+1}$ is indeed a submersion. \\

In order to prove this, let us fix $(h,u,j,J,H)\in\widetilde{\IM}^{\gamma^+}_{\gamma_0,\gamma_1}(\Gamma)$. We need to show that $$\forall y\in T_j\IM=T_j\IM_{r+1}\;\;\exists (\chi,\xi,y,Y,G)\in T_{(h,u,j,J,H)}\widetilde{\IM}^{\gamma^+}_{\gamma_0,\gamma_1}(\Gamma).$$ Note that, in the notation from \cite{F1}, this is equivalent to showing that for every $y\in T_j\IM$ there exist $\chi\in H^{1,p,d}_{\cst}(\Si,\IC)$, $\xi\in H^{1,p}(u^*TM)$, $Y\in T_J\mathcal{J}$ and $G\in T_H\mathcal{H}$ such that \begin{eqnarray*} D_{h,u,j,J,H}\cdot (\chi,\xi,y,Y,G) &=& (\CR\chi + D^1_j y,D_u\xi + D^2_j y + D_J Y + D_H G)\\&\in& L^{p,d}(\Lambda^{0,1})\otimes L^p(u^*TM).\end{eqnarray*} is equal to zero. \\

First, since for every choice of complex structure $j$ on the punctured Riemann surface $\Si$ there exists a branched covering map $h=(h_1,h_2)$ to the cylinder, it follows that for every $y\in T_j\IM$ there exists $\chi\in H^{1,p,d}_{\cst}(\Si,\IC)$ such that $\CR\chi=D^1_j y=i\cdot dh\cdot y$. Note that we use that this continues to hold true when we require that the additional marked point which is fixed by the three special punctures gets mapped to $0\in S^1$ under $h_2$. \\

On the other hand, by setting $\xi$ and $G$ equal to zero, for the second statement it suffices to show that there exists $Y\in T_J\mathcal{J}$ such that $Y\cdot du\cdot j = D_J Y = D^2_j y = J\cdot du\cdot y$. But this follows, as in \cite{We}, by observing that the latter is equivalent to $$Y(h,j,z,u(z))\cdot du(z) = du(z)\cdot y(z)$$ for all points $z\in\Si$. In particular, note that we do \emph{not} need to assume that $u$ is somewhere-injective as $Y$ can be chosen to depend not only on $u(z)$ but also on $h$, $j$ and $z$ itself.

\end{document}